\documentclass[12pt]{article}

\usepackage{amsmath,amsthm,amssymb}
\allowdisplaybreaks
\usepackage{mathrsfs}
\usepackage{cite}

\usepackage[pdftex,bookmarksopen=true,bookmarks=true,unicode,setpagesize]{hyperref}
\hypersetup{colorlinks=true,linkcolor=black,citecolor=black}

\newtheorem{theorem}{Theorem}
\newtheorem{corollary}[theorem]{Corollary}
\newtheorem{lemma}[theorem]{Lemma}
\newtheorem{proposition}[theorem]{Proposition}

\theoremstyle{remark}

\theoremstyle{remark}

\theoremstyle{remark}
\newtheorem{remark}[theorem]{Remark}

\newcommand{\R}{\mathbb R}

\begin{document}

\vspace{-20mm}
\begin{center}{\Large \bf
Gauge-invariant quasi-free states on the algebra of the anyon commutation relations
}
\end{center}

{\large Eugene Lytvynov}\\ Department of Mathematics,
Swansea University, Singleton Park, Swansea SA2 8PP, U.K.;
e-mail: \texttt{e.lytvynov@swansea.ac.uk}\vspace{2mm}


{\small

\begin{center}
{\bf Abstract}
\end{center}
\noindent Let $X=\mathbb R^2$ and let $q\in\mathbb C$, $|q|=1$.  For $x=(x^1,x^2)$ and $y=(y^1,y^2)$ from $X^2$, we define a function $Q(x,y)$ to be equal to $q$ if $x^1<y^1$, to $\bar q$ if $x^1>y^1$, and to $\Re q$ if $x^1=y^1$. Let $\partial_x^+$, $\partial_x^-$ ($x\in X$) be operator-valued distributions such that $\partial_x^+$ is the adjoint of $\partial_x^-$. We say that $\partial_x^+$, $\partial_x^-$ satisfy the anyon commutation relations (ACR) if $\partial^+_x\partial_y^+=Q(y,x)\partial_y^+\partial_x^+$ for $x\ne y$ and $\partial^-_x\partial_y^+=\delta(x-y)+Q(x,y)\partial_y^+\partial^-_x$ for $(x,y)\in X^2$. In particular, for $q=1$, the ACR become the canonical commutation relations and for $q=-1$, the ACR become the canonical anticommutation relations. We define the ACR algebra as the algebra generated by operator-valued integrals of $\partial_x^+$, $\partial_x^-$. We construct a class of gauge-invariant quasi-free states on the ACR algebra. Each state from this class is completely determined by a positive self-adjoint operator $T$ on the real space $L^2(X,dx)$ which  commutes with any operator of multiplication by a bounded function $\psi(x^1)$. In the case $\Re q<0$, the operator $T$ additionally satisfies $0\le T\le -1/\Re q$. Further, for $T=\kappa^2\mathbf 1$ ($\kappa>0$), we discuss the corresponding particle density $\rho(x):=\partial_x^+\partial_x^-$. For $\Re q\in(0,1]$, using a renormalization, we rigorously define a vacuum state on the commutative algebra generated by operator-valued integrals of $\rho(x)$. This state is given by a negative binomial point process. A scaling limit of these states as $\kappa\to\infty$ gives the gamma random measure, depending on parameter $\Re q$.

 } \vspace{2mm}

  \vspace{2mm}


\section{Preliminaries and introduction}

The main aim of this paper is to construct a class of gauge-invariant quasi-free states on the algebra of the anyon commutation relations. Let us first recall the definition of the anyon commutation relations and their representation in the anyon Fock space.

\subsection{Fock space representation of the anyon commutation relations}
\label{bit78}

Let $X:=\R^d$, let $\mathscr B(X)$ denote the Borel $\sigma$-algebra on $X$, and let $ m $ denote the Lebesgue measure on $(X,\mathscr B(X))$. We denote by
$$
\mathscr H := L^2(X,  m ), \quad \mathscr H_{\mathbb C} := L^2(X\to \mathbb{C},  m )
$$
 the $L^2$-space of real-valued, respectively complex-valued
functions on $X$. (The scalar product in $\mathscr H_{\mathbb C}$ is supposed to be linear in the first dot and antilinear in the second dot.)

Consider a function $Q:X^2\to\mathbb C$ satisfying $Q(x,y)=\overline{Q(y,x)}$ and $|Q(x,y)|=1$ for all $(x,y)\in X^2$. In 1995, Liguori and Mintchev \cite{LM2,LM} introduced the notion of a {\it generalized statistics corresponding to the function $Q$}.  Heuristically, this is a family of creation operators $\partial_x^+$ and annihilation operators $\partial_x^-$ at points $x\in X$ such that $\partial_x^+$ is the adjoint of $\partial_x^-$ and these operators satisfy the following commutation relations:
\begin{align}
\partial^+_x\partial_y^+&=Q(y,x)\partial_y^+\partial_x^+,\label{gt8}\\
\partial^-_x\partial_y^-&=Q(y,x)\partial_y^-\partial_x^-,\label{ugy8t78}\\
\partial^-_x\partial_y^+&=\delta(x-y)+Q(x,y)\partial_y^+\partial^-_x.\label{ufr576er75}
\end{align}
(Formula \eqref{ugy8t78} is, in fact, a consequence of \eqref{gt8}.)
A rigorous meaning of the operators $\partial_x^+$ and $\partial_x^-$ and the commutation relations \eqref{gt8}--\eqref{ufr576er75} is given by smearing these relations with functions from the space $\mathscr H_{\mathbb C}$.
More precisely, for any $h\in \mathscr H_{\mathbb C}$, one defines linear operators
\begin{equation}\label{furt7i}
a^+(h)=\int_X m (dx)\,h(x)\,\partial_x^+,\quad a^-(h)=\int_X m (dx)\,\overline{h(x)}\,\partial_x^-\end{equation}
on a dense linear subspace $\Theta$ of a complex Hilbert space $\mathscr G$ such that the adjoint of $a^+(h)$ restricted to $\Theta$ is $a^-(h)$, and these operators satisfy the commutation relations:
\begin{align}
a^+(g)a^+(h)&=\int_{X^2} m ^{\otimes 2}(dx\,dy)\,g(x)h(y)Q(y,x)\partial_y^+\partial_x^+,\label{trt7o}\\
a^-(g)a^-(h)&=\int_{X^2} m ^{\otimes 2}(dx\,dy)\,\overline{g(x)h(y)}\,Q(y,x)\partial_y^-\partial_x^-,\label{niuy9}\\
a^-(g)a^+(h)&=\int_{X}\overline{g(x)}\,h(x)\, m (dx)+\int_{X^2} m ^{\otimes 2}(dx\,dy)\,\overline{g(x)}\,h(y)Q(x,y)\partial_y^+\partial_x^-\label{yurt8o6}
\end{align}
for any $g,h\in \mathscr H_{\mathbb C}$. Of course, the linear operators on the right hand side of formulas \eqref{trt7o}--\eqref{yurt8o6} should be given a rigorous meaning. In the case $Q\equiv 1$, formulas \eqref{trt7o}--\eqref{yurt8o6} become the {\it canonical commutation relations} (CCR), describing bosons, while in the case   $Q\equiv -1$, formulas \eqref{trt7o}--\eqref{yurt8o6} become the {\it canonical anticommutation relations} (CAR), describing fermions. In the general case, we will call \eqref{trt7o}--\eqref{yurt8o6} the {\it $Q$-commutation relations} ($Q$-CR).

Liguori and Mintchev \cite{LM2,LM} derived a representation of the $Q$-CR in the Fock space of $Q$-symmetric functions.
By using also \cite{BLW}, we will now briefly recall this construction.

A function $f^{(n)}: X^n\to\mathbb C$ is called {\it $Q$-symmetric\/} if for any $i\in\{1,\dots,n-1\}$ and $(x_1,\dots,x_n)\in  X^n$.
\begin{equation}\label{ug8y87}
 f^{(n)}(x_1,\dots,x_n)=Q(x_i,x_{i+1})f^{(n)}(x_1,\dots,x_{i-1},x_{i+1},x_i,x_{i+2},\dots,x_n).\end{equation}
For each $n\in\mathbb N$, we have $\mathscr H_{\mathbb C}^{\otimes n}=L^2(X^n\to \mathbb{C},  m ^{\otimes n})$. We denote by $\mathscr H_{\mathbb C}^{\circledast n}$ the subspace of
$\mathscr H_{\mathbb C}^{\otimes n}$ which consists of all ($ m ^{\otimes n}$-versions of) $Q$-symmetric functions from $\mathscr H_{\mathbb C}^{\otimes n}$. We call $\mathscr H_{\mathbb C}^{\circledast n}$ the {\it $n$-th $Q$-symmetric tensor power of $\mathscr H_{\mathbb C}$}.

Consider the group $S_n$ of all permutations of $1,\dots,n$. For each $\pi\in S_n$, we define a function $Q_\pi: X^n\to\mathbb C$ by
\begin{equation}\label{higuy8a}
 Q_\pi(x_1,\dots,x_n):=\!\!\! \prod_{\substack{ 1\le i < j \le n\\[1mm]
\pi(i) > \pi(j)}}
\!\!\! Q(x_i, x_j).\end{equation}
Note that, in the case $Q\equiv1$, we get $Q_\pi\equiv1$, while in  the case $Q\equiv -1$, we get
$Q_\pi\equiv (-1)^{|\pi|}= \operatorname{sgn}\pi$. Here $|\pi|$ is the number of inversions of $\pi$, i.e., the number of $i<j$ such that $\pi(i)>\pi(j)$.

For a function $f^{(n)}: X^n\to\mathbb C$, we define its $Q$-symmetrization by
\begin{equation}\label{hgyi8}
 (P_nf^{(n)})(x_1,\dots,x_n):=\frac1{n!}\sum_{\pi\in S_n}Q_\pi(x_1,\dots,x_n) f^{(n)}(x_{\pi^{-1}(1)},\dots, x_{\pi^{-1}(n)}).\end{equation}
The operator $P_n$ determines the orthogonal projection of $\mathscr H_{\mathbb C}^{\otimes n}$ onto $\mathscr H_{\mathbb C}^{\circledast n}$. Furthermore, for any $k,n\in\mathbb N$, $k<n$, we have
\begin{equation}\label{gyut67}
P_n(P_k\otimes P_{n-k})=P_n.\end{equation}
Here $P_1$ denotes the identity operator in $\mathscr H_{\mathbb C}$. For any $f^{(n)}\in \mathscr H_{\mathbb C}^{\circledast n}$ and $g^{(m)}\in \mathscr H_{\mathbb C}^{\circledast m}$, we define the {\it $Q$-symmetric tensor product of $f^{(n)}$ and $g^{(m)}$} by
$$ f^{(n)}\circledast g^{(m)}:=P_{n+m}(f^{(n)}\otimes g^{(m)}).$$
By \eqref{gyut67}, the tensor product $\circledast$ is associative.

For a Hilbert space $H$ and a constant $c>0$, we denote by $Hc$ the Hilbert space which coincides with $H$ as a set and the tensor product in $Hc$ is equal to the tensor product in $H$ times $c$.
We define a {\it $Q$-Fock space over $\mathscr H$\/} by
$$\mathscr F^Q(\mathscr H):=\bigoplus_{n=0}^\infty \mathscr H_{\mathbb C}^{\circledast n}n!\, .$$
Here $\mathscr H_{\mathbb C}^{\circledast 0}:=\mathbb C$.
The vector $\Omega:=(1,0,0,\dots)\in \mathscr F^Q(\mathscr H)$ is called the {\it vacuum}.
We  denote by $\mathscr F_{\mathrm{fin}}^Q(\mathscr H)$ the subset of $\mathscr F^Q(\mathscr H)$ consisting of all finite sequences $$F=(f^{(0)},f^{(1)},\dots,f^{(n)},0,0,\dots)$$ in which $f^{(i)}\in \mathscr H_{\mathbb C}^{\circledast i}$ for $i=0,1,\dots,n$, $n\in\mathbb N$. This space can be endowed with the topology of the topological direct sum of the $\mathscr H_{\mathbb C}^{\circledast n}$ spaces. Thus, convergence in $\mathcal F_{\mathrm{fin}}^Q(\mathcal H)$ means uniform finiteness of non-zero components and coordinate-wise convergence in $\mathscr H_{\mathbb C}^{\circledast n}$.

For each $h\in\mathscr H_{\mathbb C}$, we define a {\it creation operator\/} $a^+(h)$ and an {\it annihilation operator\/} $a^-(h)$ as linear operators acting on $\mathscr F_{\mathrm{fin}}^Q(\mathscr H)$ that satisfy
\begin{equation}\label{bhyut76}
a^{+}(h)f^{(n)} := h \circledast f^{(n)},\quad f^{(n)}\in \mathscr H_{\mathbb C}^{\circledast n},\quad a^{-}(h):= (a^{+}(h))^*\restriction _{\mathscr F_{\mathrm{fin}}^Q(\mathscr H)}.
\end{equation}
These operators act continuously on $\mathscr F_{\mathrm{fin}}^Q(\mathscr H)$. Furthermore, for $h\in \mathscr H_{\mathbb C}$ and $f^{(n)}\in \mathscr H_{\mathbb C}^{\circledast n}$, we
have:
\begin{equation}\label{bhgyt8}
(a^-(h)f^{(n)})(x_1, \dots , x_{n-1}) =
 n\int_{X}\overline{h(y)}\,f^{(n)}(y, x_1, \dots ,
x_{n-1})\, m (dy).\end{equation}
Thus, if we introduce informal operators $\partial_x^+$ and $\partial_x^-$ by formulas \eqref{furt7i}, we get, for $f^{(n)}\in\mathscr H_{\mathbb C}^{\circledast n}$,
$$\partial_x^+ f^{(n)}=\delta_x\circledast f^{(n)},\qquad \partial_x^- f^{(n)}=nf^{(n)}(x,\cdot).$$
where $\delta_x$ is the delta function at $x$.
Now, one can easily give a rigorous meaning to the operators on the right hand side of formulas \eqref{trt7o}--\eqref{yurt8o6} and show that  the $Q$-CR hold.

We note that, in the obtained representation of the $Q$-CR, we only used the values of the function $Q$ $ m ^{\otimes 2}$-almost everywhere.
Hence, for this representation, we could assume from the very beginning that there exists a set $\Delta\in\mathscr B(X^2)$ which is symmetric (i.e., if $(x,y)\in \Delta$, then $(y,x)\in \Delta$) and  satisfies $ m ^{\otimes 2}(\Delta)=0$, and the function $Q$ is only defined on the set $ \widetilde X^{2}:=X^2\setminus \Delta$. Since the measure $ m $ is non-atomic, we may also assume that $D\subset\Delta$, where $D:=\{(x,x)\mid x\in X\}$ is the diagonal in $X^2$.

In physics, intermediate statistics have been discussed since
Leinass and Myrheim \cite{Leinaas_Myrheim} conjectured their existence in 1977.   The first mathematically rigorous prediction of intermediate statistics was done by Goldin, Menikoff and Sharp \cite{GMS1,GMS2} in 1980, 1981. The name {\it anyon} was given to such statistics by Wilczek \cite{W1,W2}. Anyon statistics were used, in particular, to describe the quantum Hall effect, see e.g.\ the review paper \cite{Stern}.

Liguori, Mintchev \cite{LM2,LM} and Goldin, Sharp   \cite{Goldin_Sharp} showed that anyon statistics can be described by the $Q$-CR in which $X=\R^2$, the set $\Delta$ is chosen as
\begin{equation}\label{tye56}
\Delta:=\big\{(x,y)\in X^2\mid x^1=y^1\big\}\end{equation}
 and
\begin{equation}\label{tye6}
 Q(x,y)=\begin{cases}q,&\text{if }x^1<y^1,\\
 \bar q,&\text{if }x^1>y^1.\end{cases}\end{equation}
Here, $q\in\mathbb C$ with $|q|=1$, and for $x\in X$ we denote by $x^i$ the $i$th coordinate of $x$. With such a choice of the function $Q$, formulas \eqref{trt7o}--\eqref{yurt8o6} are called the {\it anyon commutation relations\/} (ACR).
We note that Goldin, Sharp \cite{Goldin_Sharp} realized the ACR by using  operators acting on the space of functions of finite configurations in $X$ (or, equivalently, in the symmetric Fock space).

Goldin and  Majid  \cite{GM} showed that, in the case where $q$ is a $k$th root of 1 and $q\ne1$, the corresponding statistics  satisfies the natural anyonic exclusion principle, which generalizes  Pauli's exclusion principle for fermions:
\begin{equation}\label{vut}
a^+(f)^k=0\quad \text{for each }f\in\mathscr H_{\mathbb C}.\end{equation}

For further discussions of anyons in mathematical physics literature (including the discrete setting), see e.g.\ \cite{DAFT,DK,GM,FSSS,LS,LM3,LMP,LMR0,LMR1,LMR2,LMR3} and the references therein.  We also refer to the paper \cite{BS} which deals with a Fock representation of the commutations relations identified by a sequence  of self-adjoint operators in a Hilbert space which have norm $\le1$ and which satisfy the braid relations.

\subsection{Gauge-invariant quasi-free states on the CCR and CAR algebras}

In the theory of the CCR and CAR algebras,  quasi-free states, in particular, gauge-invariant quasi-free states, play a fundamental role. We refer the reader to e.g.\ Sections~5.2.1--5.2.3 and Notes and Remarks to these sections in \cite{BR}, and  \cite[Chapter~17]{DG}, see also the pioneering book \cite[Chapter II]{Berezin1} and paper \cite{Berezin2}. We  note that gauge-invariant quasi-free states describe, in particular,  the infinite free Bose gas at finite temperature \cite{AWoods} (see also \cite{DA} and Section~5.2.5 in \cite{BR}) and the infinite free Fermi gas at both finite and zero temperatures \cite{AWyss} (see also \cite{DA} and Section~5.2.4 in \cite{BR}). Free analogs of quasi-free states have been discussed in \cite{Shl}, see also \cite{Hiai}.

Let us recall that the CCR algebra (or the CAR algebra), $\mathbf A$, is a complex algebra  generated by linear operators $a^+(h)$, $a^-(h)$ ($h\in\mathscr H_{\mathbb C}$) satisfying the CCR (the CAR, respectively). Because of the commutation relations, each element of $\mathbf A$ can be represented as a finite sum  of a constant and  operators
$$a^{\sharp_1}(h_1)\dotsm a^{\sharp_n}(h_k),\quad h_1,\dots,h_k\in\mathscr H_{\mathbb C},\ \sharp_1,\dots,\sharp_k\in\{+,-\},$$
which are in the {\it Wick order\/}. The latter means that there is  no $i\in\{1,\dots,k-1\}$ such that $\sharp_i=-$ and $\sharp_{i+1}=+$, i.e.,  there is no creation operator acting before an annihilation operator.

Let $\tau$ be a state on the algebra $\mathbf A$. One defines {\it $n$-point functions\/} by
\begin{equation}\label{tuyr75i}
\mathbf S^{(k,n)}(g_k,\dots,g_1,h_1,\dots,h_n):=\tau\big(a^+(g_k)\dotsm a^+(g_1)a^-(h_1)\dotsm a^-(h_n)\big), \end{equation}
where $g_1,\dots,g_k,h_1,\dots,h_n\in\mathscr H_{\mathbb C}$ and $k,n\in\mathbb N$.
One says that the state $\tau$ is {\it gauge-invariant} if it is invariant under the group of Bogoliubov transformations
$$a^+(h)\mapsto a^+(e^{i\theta}h)=e^{i\theta}a^+(h),\quad
a^-(h)\mapsto a^-(e^{i\theta}h)=e^{-i\theta}a^-(h),\quad \theta\in[0,2\pi).$$
 By \eqref{tuyr75i}, $\tau$ is gauge-invariant if and only if $\mathbf S^{(k,n)}\equiv0$ for $k\ne n$. Thus, a gauge-invariant state is completely determined by $\mathbf S^{(n,n)}$ ($n\in\mathbb N$).

A state $\tau$ is called a {\it gauge-invariant quasi-free state} if $\mathbf S^{(k,n)}\equiv0$ for $k\ne n$ and the $n$-point functions $\mathbf S^{(n,n)}$ are completely determined by $\mathbf S^{(1,1)}$. More precisely, in the case of the CCR algebra, we have
\begin{equation}\label{cfyd6}
\mathbf S^{(n,n)}(g_n,\dots,g_1,h_1,\dots,h_n)=\operatorname{per}\left[\mathbf S^{(1,1)}(g_i,h_j)\right]=\sum_{\pi\in S_n}\prod_{i=1}^n\mathbf S^{(1,1)}(g_i,h_{\pi(i)}),
\end{equation}
and in the case of the CAR algebra, we have
\begin{equation}\label{iudgsy}
\mathbf S^{(n,n)}(g_n,\dots,g_1,h_1,\dots,h_n)=\operatorname{det}\left[\mathbf S^{(1,1)}(g_i,h_j)\right]=\sum_{\pi\in S_n} \operatorname{sgn}\pi\prod_{i=1}^n\mathbf S^{(1,1)}(g_i,h_{\pi(i)}).
\end{equation}

A gauge-invariant quasi-free state on the CCR algebra is completely identified by a bounded linear  operator $T$ in $\mathscr H_{\mathbb C}$, with $T\ge0$, which satisfies
\begin{equation}\label{yiut}
\mathbf S^{(1,1)}(g,h)=(Tg,h)_{\mathscr H_{\mathbb C}}.\end{equation}
Respectively, a  gauge-invariant quasi-free state on the CAR algebra is completely identified by a bounded linear  operator $T$ in $\mathscr H_{\mathbb C}$, with $0\le T\le  1$, which satisfies
\eqref{yiut}.

The corresponding representation of the CCR/CAR algebra  can be given on the symmetric/antisymetric Fock space over $\mathscr H\oplus \mathscr H$ by using the bounded linear operators $\sqrt T$ and $\sqrt{1+T}$ in the CCR case and  $\sqrt T$ and $\sqrt{1-T}$ in the CAR case, see e.g.\ Examples~5.2.18 and 5.2.20 in \cite{BR}.

\subsection{A brief description of the results}\label{vufyut}

While our main interest in this paper will be the ACR, we will actually deal with a slightly
more general form of the $Q$-CR: we will assume that
$X=\R^d$ with $d\ge 2$ and the function $Q: \widetilde X^{2}\to\mathbb C$ satisfies $Q(x,y)=Q(x^1,y^1)$ (with an obvious abuse of notation).
Here $ \widetilde X^{2}:=X^2\setminus \Delta$ with $\Delta$ being given by \eqref{tye56}.

We saw in the Fock space representation that defining a function $Q$ on $\widetilde X^{2}$ was enough. However,
 we will see below that, in the general case,  this is not enough for relation \eqref{ufr576er75} and we need to specify the values of $Q(x,x)$ for $x\in X$. In the case of the bose and fermi statistics, we take, of course,  $Q(x,x)\equiv 1$ and $Q(x,x)\equiv -1$, respectively.

So from now on we will assume that, for some constant $\eta\in\R$, we have $Q(x,y)=\eta$ for all $(x,y)\in\Delta$, in particular, $Q(x,x)=\eta$ for all $x\in X$. 
We will define a $Q$-CR algebra so that the value $\eta$ will matter for relation \eqref{ufr576er75}, but $\eta$ will
be of no importance to  relations \eqref{gt8},     \eqref{ugy8t78} as they will still depend on the values of the function $Q$ $ m ^{\otimes 2}$-almost everywhere.

 We see that, in the anyon case (with $q\ne\pm1$), the function $Q$ cannot be extended to a continuous function on $X^2$, so there is a freedom in choosing the value of $\eta$.
A natural choice for $\eta$ seems to be  $\eta=\Re(q)=(q+\bar q)/2$.

The form of the $Q$-CR means that it is not enough to consider a complex algebra generated by the operators \eqref{furt7i}. Instead, in Section~\ref{guftk7r}, we  consider a complex algebra $\mathbf A$  generated by operator-valued integrals
\begin{equation}\label{yuur}
\int_{X^k} m ^{\otimes k}(dx_1\dotsm dx_k)\,\varphi^{(k)}(x_1,\dots,x_k)\,\partial_{x_1}^{\sharp_1}\dotsm\partial_{x_k}^{\sharp_k},\end{equation}
where the class of functions $\varphi^{(k)}:X^k\to\mathbb C$
appearing in the integral \eqref{yuur} will be specified.
We  show that the anyon exclusion principle (see \eqref{vut})
holds in the general ACR algebra  for $q$ being a root of 1, $q\ne1$.

If $\tau$ is a state
on $\mathbf A$, then due to \eqref{ufr576er75}, $\tau$ is completely determined by the $n$-point functions
\begin{align}& \mathbf S^{(k,n)}(\varphi^{(k+n)}):=\tau\bigg(
\int_{X^{k+n}} m ^{\otimes(k+n)}(dx_1\dotsm dx_{k+n})\,\varphi^{(k+n)}(x_1,\dots,x_{k+n}) \notag\\
&\quad\times \partial_{x_1}^+\dotsm \partial_{x_k}^+\partial_{x_{k+1}}^-\dotsm \partial_{x_{k+n}}^-
\bigg).\label{byut678}\end{align}

We  say that the state $\tau$  is {\it gauge-invariant} if it is invariant under the group of Bogoliubov transformations $\partial^+_x\mapsto e^{i\theta}\partial^+_x$, $\partial^-_x\mapsto e^{-i\theta}\partial^-_x$, with $\theta\in[0,2\pi)$, or equivalently if $\mathbf S^{(k,n)}\equiv0$ for $k\ne n$.

So it is intuitively clear what it should mean that $\tau $ is a {\it gauge-invariant quasi-free state}: we should have
$\mathbf S^{(k,n)}\equiv 0$ if $k\ne n$ and the $n$-point functions $\mathbf S^{(n,n)}$ should be completely determined by $\mathbf S^{(1,1)}$. However, to write down a proper generalization of formulas \eqref{cfyd6}, \eqref{iudgsy} is not straightforward:
instead of the sign of a permutation $\pi$ we should use the function $Q_\pi$ (see \eqref{higuy8a}), and the functions $\varphi^{(k+n)}$ appearing in \eqref{byut678} do not necessarily factorize to separate their  variables. We solve this problem in Section~\ref{guftk7r} by properly introducing certain measures on $\R^{2n}$ corresponding to the $n$-point functions. As a result, the definition of a gauge-invariant quasi-free state for the $Q$-CR generalizes the available definitions in the CCR and CAR cases.

In Section~\ref{jhuf76} we construct operator-valued integrals  \eqref{yuur} in the $Q$-symmetric Fock space. The presentation in this section is given at a rather general level. In particular, in this section we assume that $X$ is a locally compact Polish space, while $ m $ is a non-atomic Radon measure on $X$.

In Section~\ref{uiy977}, we  construct a representation of the $Q$-CR algebra $\mathbf A$ and a class of  gauge-invariant quasi-free states $\tau$ on it.  This construction is done in a $JQ$-symmetric Fock space  over $\mathscr H\oplus\mathscr H$. Here $JQ:Z^2\to\mathbb C$ is a function on the space $Z:=X_1\sqcup X_2$, the disjoint union of two  copies of $X$. Explicitly, the function $JQ$ is defined through the function $Q$ by formula \eqref{oyyoy} below.

The operator $T$, being defined analogously to formula \eqref{yiut}, satisfies in our setting  the following assumptions:

\begin{itemize}

\item $T$ is a self-adjoint bounded linear operator in the real space $\mathscr H$ and is extended by linearity to $\mathscr H_{\mathbb C}$;

\item  $T$ commutes with any operator of multiplication by a bounded function $\psi(x^1)$;

\item in the  case $\eta\ge0$, we have $T\ge0$, and in the case $\eta<0$, we have $0\le T\le -1/\eta$.

\end{itemize}

For
$$\varphi^{(2n)}(x_1,\dots,x_{2n})=g_n(x_1)\dotsm g_1(x_{n})h_1(x_{n+1})\dotsm h_n(x_{2n}) $$
with $g_1,\dots,g_n,h_1,\dots,h_n\in\mathscr H_{\mathbb C}$, we obtain
$$
\mathbf S^{(n,n)}(\varphi^{(2n)})=\sum_{\pi\in S_n}
\int_{X^n}\left(\prod_{i=1}^n g_i(x_i)(Th_{\pi(i)})(x_i)\right)Q_\pi(x_1,\dots,x_n)\, m ^{\otimes n}(dx_1\dotsm dx_n),
$$
compare with \eqref{cfyd6}--\eqref{yiut}.

Finally, in Section~\ref{crds54w}, we discuss the particle density   associated with a gauge-invariant quasi-free state on the ACR algebra with $\eta=\Re q$. The particle density is informally defined by $\rho(x):=\partial_x^+\partial_x^-$ for $x\in X$. It follows from the ACR that these operators commute, cf.\ \cite{GM,Goldin_Sharp}. Hence, the state on the algebra generated by the commutative operators $\rho(f):=\int_X  m (dx)\,f(x)\rho(x)$ ($f$ running through a space of test functions on $X$) should be given by a probability measure $\mu$.

In the case of the CCR and CAR algebras, it was shown in \cite{Ly,LyMei} that, for $T$ being a locally trace class operator, $\mu$ is a permanental (determinantal, respectively) point process on $X$.  Note, however, that our assumptions on the operator $T$ exclude locally trace class operators.

In this paper, we treat the case where $T$ is a constant operator, $T=\kappa^2\mathbf 1$ with $\kappa>0$. Under this assumption, it is not possible to give a rigorous meaning to $\rho(f)$ as a self-adjoint linear operator in the $JQ$-symmetric Fock space over $\mathscr H\oplus\mathscr H$. So a renormalization is needed. Similarly to the construction of the renormalized square of white noise algebra \cite{AFS,ALV}, as renormalization we will use Ivanov's formula \cite{Ivanov} which suggests that the square of the delta function, $\delta^2$, can be interpreted as $c\delta$, where $c$ is any positive constant. For our calculations, we choose $c=1$, so that Ivanov's formula becomes $\delta^2=\delta$. We note that a different choice of the constant $c$ would lead to similar results in which the measure $ m $ is replaced by $c m $. 

So, using this renormalization and the  ACR,  we rigorously define a
functional $\tau$ on the algebra generated by commutative operators $\rho(f)$.
However, due to renormalization, it is not {\it a priori\/} clear whether $\tau$ is a state, i.e., whether it is positive definite. We prove that $\tau$ is indeed a state if and only if  $\eta\in[0,1]$.

Furthermore, for $\eta=0$, the state $\tau$ is given by the Poisson point process on $X$ with intensity measure $\kappa^2 m $, while for $\eta>0$, $\tau$ is given by a negative binomial point process on $X$, which depends on two parameters, $\eta$ and $\kappa$. The latter process takes values in the space $\ddot\Gamma(X)$ of multiple configurations in $X$, i.e.,  Radon measures on $X$ which take values in $\{0,1,2,\dots,\infty\}$. Note also that the negative binomial point process is a measure-valued L\'evy process on $X$ whose L\'evy measure is finite. Finally, we prove that, for a fixed $\eta>0$,  a (scaling) limit of the states depending on $\kappa$  exists as $\kappa\to\infty$, and the limiting state is
 given by  the gamma random measure,  depending on parameter $\eta$. This random measure is known to have many distinguished properties, see e.g.\   \cite{VGG1,VGG3,VGG2,KLV,KSSU,KL,TsVY}.  We stress that the results of Section~\ref{crds54w} are new even in the CCR case.

\section{The $Q$-CR algebra and gauge-invariant quasi-free states on it}\label{guftk7r}

\subsection{Preliminary  definitions}\label{hvjgyu}

In this section, we assume that $X=\R^d$ with $d\ge2$, $ m $ is the Lebesgue measure on it,  $Q(x,y)=Q(x^1,y^1)$ for  $(x,y)\in \widetilde X^{2}$, and $Q(x,y)=\eta$ for $(x,y)\in\Delta$. To define the $Q$-CR algebra, we need an appropriate class of functions $\varphi^{(k)}$ appearing in \eqref{yuur}.

Let $k\in\mathbb N$. We denote  by $\Pi(k)$ the set of all partitions of the set $\{1,\dots,k\}$, i.e., all collections of mutually disjoint sets whose union is $\{1,\dots,k\}$. For each partition
$\theta\in\Pi(k)$, we denote by $X^{(k)}_\theta$ the subset of $X^k$ which consists of all $(x_1,\dots,x_k)\in X^k$ such that, for all $1\le i<j\le k$, $x_i=x_j$ if and only if $i$ and $j$ belong to the same element of the partition $\theta$. Note that the sets $X^{(k)}_\theta$ with $\theta\in \Pi(k)$
form a partition of $X^k$. We denote $X^{(k)}:=X_\theta^{(k)}$ for the minimal partition $\theta=\big\{\{1\},\,\{2\},\dots,\, \{k\}\big\}$.

Let $\theta=\{\theta_1,\dots,\theta_l\}\in\Pi(k)$ and assume that
$$\min \theta_1<\min \theta_2<\dots<\min \theta_l.$$
We have $m^{\otimes l}(X^l\setminus X^{(l)})=0$, so we can consider $m^{\otimes l}$ as a measure on $X^{(l)}$.  Consider the mapping $I_\theta:X^{(l)}\to X^{(k)}_\theta$ given by $I_\theta(x_1,\dots,x_l)=(y_1,\dots,y_k)$, where for each $i\in\theta_1$ we have $y_i=x_1$, for each $i\in\theta_2$ we have $y_i=x_2$, etc. We denote by $m^{(k)}_\theta$ the pushforward of the measure $m^{\otimes l}$ under $I_\theta$. We extend the measure $m^{(k)}_\theta$ by zero to the whole space $X^k$. Note that $m^{(k)}=m^{\otimes k}_\theta $ for the minimal partition $\theta=\big\{\{1\},\,\{2\},\dots,\, \{k\}\big\}$.

Let us fix any $\sharp_1,\dots,\sharp_k\in\{+,-\}$. We
denote by $\Pi(k,\sharp_1,\dots,\sharp_k)$ the subset of $\Pi(k)$ which consists of all partitions $\theta=\{\theta_1,\dots,\theta_l\}$ such that each set $\theta_i$ has at most two elements, and if $\theta_i=\{a,b\}$ has two elements then $\sharp_a\ne\sharp_b$. We  define a measure on $X^k$ by
$$m^{(k)}_{\sharp_1,\dots,\sharp_k}:=\sum_{\theta\in \Pi(k,\sharp_1,\dots,\sharp_k)}m^{(k)}_\theta.$$
For example, for a measurable function $f:X^3\to[0,\infty)$, we have
\begin{align}&\int_{X^3}f(x_1,x_2,x_3)\, m ^{(3)}_{-,-,+}(dx_1\,dx_2\,dx_3)=
\int_{X^3}f(x_1,x_2,x_3)\, m ^{\otimes 3}(dx_1\,dx_2\,dx_3)\notag\\
&\quad+\int_{X^2}f(x_1,x_2,x_1)\, m ^{\otimes 2}(dx_1\,dx_2)+\int_{X^2}f(x_1,x_2,x_2)\, m ^{\otimes 2}(dx_1\,dx_2).\label{ye6e6}\end{align}

Completely analogously, starting with the Lebesgue measure on $\R$ rather than $X$, we define a measure $\nu^{(k)}_{\sharp_1,\dots,\sharp_k}$ on $\R^{k}$. For example, similarly to \eqref{ye6e6}, for a measurable function $f:\R^3\to[0,\infty)$, we have
\begin{align*}&\int_{\R^3}f(s_1,s_2,s_3)\,\nu^{(3)}_{-,-,+}(ds_1\,ds_2\,ds_3)=
\int_{\R^3}f(s_1,s_2,s_3)\,ds_1\,ds_2\,ds_3\\
&\quad+\int_{\R^2}f(s_1,s_2,s_1)\,ds_1\,ds_2+\int_{\R^2}f(s_1,s_2,s_2)\,ds_1\,ds_2.\end{align*}

We denote by $L^0(X^k\to\mathbb C, m^{(k)}_{\sharp_1,\dots,\sharp_k})$
the linear space of classes of complex-valued measurable functions on $X^k$, with any two measurable functions $f,g:X^k\to\mathbb C$ being identified if $f=g$ $m^{(k)}_{\sharp_1,\dots,\sharp_k}$-a.e. We define a linear mapping
\begin{equation}\label{tye7i554o7}
\Phi^{(k)}_{\sharp_1,\dots,\sharp_k}: \mathscr H_{\mathbb C}^k\times L^\infty(\R^k\to\mathbb C,\nu^{(k)}_{\sharp_1,\dots,\sharp_k})\to L^0(X^k\to\mathbb C, m^{(k)}_{\sharp_1,\dots,\sharp_k})\end{equation} by
\begin{equation}\label{vugf7u}
\Phi^{(k)}_{\sharp_1,\dots,\sharp_k}\big[h_1,\dots,h_k,v^{(k)}\big] (x_1,\dots,x_k)
:=h_1(x_1)\dotsm h_k(x_k)v^{(k)}(x_1^1,\dots,x_k^1),
\end{equation}
where $h_1,\dots,h_k\in\mathscr H_{\mathbb C}$ and $v^{(k)}\in L^\infty(\R^k\to\mathbb C,\nu^{(k)}_{\sharp_1,\dots,\sharp_k})$.
We denote by $\mathbb F^{(k)}_{\sharp_1,\dots,\sharp_k}$ the range of  $\Phi^{(k)}_{\sharp_1,\dots,\sharp_k}$, which is a subspace of $L^0(X^k\to\mathbb C, m^{(k)}_{\sharp_1,\dots,\sharp_k})$.

\begin{remark}
It should be noted that the mapping $\Phi^{(k)}_{\sharp_1,\dots,\sharp_k}$ is not injective. For example, if $h_1,\dots,h_k\in\mathscr H_{\mathbb C}$, $v^{(k)}\in L^\infty(\R^k\to\mathbb C,\nu^{(k)}_{\sharp_1,\dots,\sharp_k})$, and $\alpha\in L^\infty (\R,ds)$, we have
$$\Phi^{(k)}_{\sharp_1,\dots,\sharp_k}\big[g,h_2\dots,h_k,v^{(k)}\big]=\Phi^{(k)}_{\sharp_1,\dots,\sharp_k}\big[h_1,\dots,h_k,w^{(k)}\big],$$
where $g(x):=h_1(x)\alpha(x^1)\in\mathscr H_{\mathbb C}$ and $w^{(k)}(s_1,\dots,s_k):=v^{(k)}(s_1,\dots,s_k)\alpha(s_1)\in \linebreak L^\infty(\R^k\to\mathbb C,\nu^{(k)}_{\sharp_1,\dots,\sharp_k})$.
\end{remark}

Below we will deal with linear operators in a complex Hilbert space which will be denoted by operator-valued integrals of the form
\eqref{yuur} with \begin{equation}\label{dyftyd}
\varphi^{(k)}= \Phi^{(k)}_{\sharp_1,\dots,\sharp_k}\big[h_1,\dots,h_k,v^{(k)}\big]\in \mathbb F^{(k)}_{\sharp_1,\dots,\sharp_k}.
\end{equation}
We will also denote these operators by
$I^{(k)}_{\sharp_1,\dots,\sharp_k}\big[h_1,\dots,h_k,v^{(k)}\big]$.
Our next aim is to give a rigorous definition of the commutation relations
\eqref{gt8}--\eqref{ufr576er75} satisfied by these operators.

 Let $k\ge2$, $\sharp_1,\dots,\sharp_k\in\{+,-\}$, and let $i\in\{1,\dots,k-1\}$. Let us consider the operator-valued integral  \eqref{yuur} with $\varphi^{(k)}$ given by \eqref{dyftyd}. Assume that $\sharp_i=\sharp_{i+1}$. Then, at least informally, we calculate using either relation \eqref{gt8} or relation \eqref{ugy8t78}:
 \begin{align}
&\int_{X^k} m ^{\otimes k}(dx_1\dotsm dx_k)\,\varphi^{(k)}(x_1,\dots,x_k)\partial_{x_1}^{\sharp_1} \dotsm \partial_{x_k}^{\sharp_k}\notag\\
&\quad=\int_{X^k} m ^{\otimes k}(dx_1\dotsm dx_k)\,h_1(x_1)\dotsm h_k(x_k)v^{(k)}(x^1_1,\dots,x^1_k) Q(x_{i+1}^1,x_i^1)\notag\\
&\qquad\times \partial_{x_1}^{\sharp_1}\dotsm \partial_{x_{i-1}}^{\sharp_{i-1}}\partial_{x_{i+1}}^{\sharp_{i+1}}\partial_{x_i}^{\sharp_i}\partial_{x_{i+2}}^{\sharp_{i+2}}\dotsm \partial_{x_k}^{\sharp_k}\notag\\
&\quad =\int_{X^k} m ^{\otimes k}(dx_1\dotsm dx_k)\,h_1(x_1)\dotsm
h_{i-1}(x_{i-1})h_{i+1}(x_i)h_i(x_{i+1})h_{i+2}(x_{i+2})\dotsm h_k(x_k)\notag\\
&\qquad\times (\Psi_i v^{(k)})(x_1^1,\dots,x_k^1)\partial_{x_1}^{\sharp_1}
\dotsm \partial_{x_k}^{\sharp_k}.\label{uyt68o}
\end{align}
 Here
\begin{align}
& (\Psi_i v^{(k)})(s_1,\dots,s_k)\notag\\
&\quad:=v^{(k)}(s_1,\dots,s_{i-1},s_{i+1},s_i,s_{i+2},\dots,s_k)Q(s_i,s_{i+1})\in L^\infty(\R^k\to\mathbb C,\nu^{(k)}_{\sharp_1,\dots,\sharp_k}).\label{utr67o}\end{align}

Thus, inspired by \eqref{uyt68o} and  \eqref{utr67o}, we give the following definition: Relation \eqref{gt8} (or relation \eqref{ugy8t78}, respectively) means that, for any $k\ge2$, $i\in\{1,\dots,k-1\}$ and $\sharp_1,\dots,\sharp_k\in\{+,-\}$ such that $\sharp_i=\sharp_{i+1}=+$ (or  $\sharp_i=\sharp_{i+1}=-$, respectively), we have
\begin{equation}\label{gfr7}
I^{(k)}_{\sharp_1,\dots,\sharp_k}\big[h_1,\dots,h_k,v^{(k)}\big]=
I^{(k)}_{\sharp_1,\dots,\sharp_k}\big[h_1,\dots,h_{i-1},h_{i+1},h_i,h_{i+2},\dots,h_k,\Psi_i v^{(k)}\big].\end{equation}
Analogously, relation \eqref{ufr576er75} means that  for any $k\ge2$, $i\in\{1,\dots,k-1\}$ and $\sharp_1,\dots,\sharp_k\in\{+,-\}$ such that $\sharp_i=-$ and  $\sharp_{i+1}=+$, we have
\begin{align}
&I^{(k)}_{\sharp_1,\dots,\sharp_k}\big[h_1,\dots,h_k,v^{(k)}\big]\notag\\
&\quad=I^{(k-2)}_{\sharp_1,\dots,\sharp_{i-1},\sharp_{i+2},\dots,\sharp_k}
\big[h_1,\dots,h_{i-1},h_{i+2},\dots,h_k, u^{(k-2)}\big]
\notag\\
&\qquad+ I^{(k)}_{\sharp_1,\dots,\sharp_{i-1},\sharp_{i+1},\sharp_i,\sharp_{i+2},\dots,\sharp_k}
\big[h_1,\dots,h_{i-1},h_{i+1},h_i,h_{i+2},\dots,h_k,\Psi'_i v^{(k)}\big],\label{erererertyr75}
\end{align}
where
\begin{align}
&u^{(k-2)}(s_1,\dots,s_{k-2}):=\int_X h_i(x)h_{i+1}(x)\notag\\
&\quad \times v^{(k)}(s_1,\dots,s_{i-1},x^1,x^1,s_i,\dots,s_{k-2})\, m (dx)\in L^\infty(\R^{k-2}\to\mathbb C,\nu^{(k-2)}_{\sharp_1,\dots,\sharp_{i-1},\sharp_{i+2},\dots,\sharp_k})\label{uyt68686}\end{align}
and
\begin{align}
& (\Psi'_i v^{(k)})(s_1,\dots,s_k):=v^{(k)}(s_1,\dots,s_{i-1},s_{i+1},s_i,s_{i+2},\dots,s_k)\notag\\
&\quad\times Q(s_{i+1},s_i)\in L^\infty(\R^k\to\mathbb C,\nu^{(k)}_{\sharp_1,\dots,\sharp_{i-1},\sharp_{i+1},\sharp_i,\sharp_{i+2},\dots,\sharp_k}).\label{ufr7r}\end{align}

\begin{remark}In the case $k=2$, the second addend on the right hand side of equality \eqref{erererertyr75} is understood as the constant operator
$u^{(0)}$, where
$$u^{(0)}:=\int_X h_1(x)h_2(x)v^{(2)}(x^1,x^1)\, m (dx).$$
\end{remark}

\begin{remark} Note that the commutation relations \eqref{gfr7}, \eqref{erererertyr75} do not depend on the representation of $\varphi^{(k)}\in  \mathbb F^{(k)}_{\sharp_1,\dots,\sharp_k}$
in the form \eqref{dyftyd}.

\end{remark}

\begin{remark} Note that the commutation relations \eqref{gfr7} do not depend on $\eta$. Indeed, for $\sharp_i=\sharp_{i+1}$, we have
\begin{equation}\label{vgfuyft86r6}
 \nu^{(k)}_{\sharp_1,\dots,\sharp_k}\big(\{(s_1,\dots,s_k)\mid s_i=s_{i+1}\}\big)=0.\end{equation}
Hence, in \eqref{utr67o}, for $s_i=s_{i+1}$ the value $Q(s_i,s_{i+1})=\eta$ plays no role. On the other hand, formula \eqref{vgfuyft86r6} is not true when $\sharp_i=-$ and $\sharp_{i+1}=+$. Therefore, for $s_i=s_{i+1}$ the value $Q(s_{i+1},s_i)=\eta$ does matter for \eqref{ufr7r}, hence also for the commutation relation \eqref{erererertyr75}.
\end{remark}

 \subsection{Definition of the $Q$-CR algebra and the anyon exclusion principle}

We are now in position to  define the $Q$-CR algebra. Let $\mathscr G$ be a separable, complex Hilbert space. Let $\Theta$ be a dense linear subspace of $\mathscr G$. We assume that, for any $\sharp_1,\dots,\sharp_k\in\{+,-\}$ and any $\varphi^{(k)}\in\mathbb F^{(k)}_{\sharp_1,\dots,\sharp_k}$ we have a linear  operator mapping $\Theta$ into $\Theta$. This operator is denoted either as in
\eqref{yuur} or by
$I^{(k)}_{\sharp_1,\dots,\sharp_k}\big(h_1,\dots,h_k,v^{(k)}\big)$,  given that $\varphi^{(k)}$ is as in \eqref{dyftyd}. These operators will be called {\it operator-valued integrals}.

We will assume that the operator-valued integrals satisfy the following axioms.

\begin{itemize}

\item[(A1)] {\it Consistency condition\/}:
For any $g_1,\dots,g_k,h_1,\dots,h_k\in\mathscr H_{\mathbb C}$, and $v^{(k)},w^{(k)}\in L^\infty(\R^k\to\mathbb C,\nu^{(k)}_{\sharp_1,\dots,\sharp_k})$,  if
$$\Phi^{(k)}_{\sharp_1,\dots,\sharp_k}\big[g_1,\dots,g_k,w^{(k)}\big]
=\Phi^{(k)}_{\sharp_1,\dots,\sharp_k}\big[h_1,\dots,h_k,v^{(k)}\big],$$
then
$$I^{(k)}_{\sharp_1,\dots,\sharp_k}(g_1,\dots g_k,w^{(k)})=I^{(k)}_{\sharp_1,\dots,\sharp_k}(h_1,\dots h_k,v^{(k)})
.$$

\item[(A2)] {\it Linearity:} For any $\sharp_1,\dots,\sharp_k\in\{+,-\}$,
$
I^{(k)}_{\sharp_1,\dots,\sharp_k}(h_1,\dots,h_k,v^{(k)})$
linearly depends on $h_i\in\mathscr H_{\mathbb C}$ $(i=1,\dots,k)$ and on $v^{(k)}\in L^\infty(\R^k\to\mathbb C,\nu^{(k)}_{\sharp_1,\dots,\sharp_k})$.

\item[(A3)] {\it The adjoint operator:} The adjoint of any operator
$
I^{(k)}_{\sharp_1,\dots,\sharp_k}(h_1,\dots,h_k,v^{(k)})$
 in the Hilbert space $\mathscr G$ contains $\Theta$ in its domain, and the restriction of this adjoint  operator to $\Theta$ is equal to the operator
$I^{(k)}_{\triangle_k,\dots,\triangle_1}(\overline{h_k},\dots,\overline{h_1},v^{(k)}{}^*)$,
where
\begin{align*}
&\triangle_i:=\begin{cases}+&\text{if }\sharp_i=-,\\
-&\text{if }\sharp_i=+,\end{cases}\qquad i=1,\dots,k,\\
&v^{(k)}{}^*(s_1,\dots,s_k)=\overline{v^{(k)}(s_k,\dots,s_1)}\in
L^\infty(\R^k\to\mathbb C,\nu^{(k)}_{\triangle_k,\dots,\triangle_1}).\end{align*}

\item[(A4)] {\it $Q$-commutation relations:} The operators $\partial_x^+$, $\partial_x^-$ satisfy the $Q$-CR
\eqref{gt8}--\eqref{ufr576er75}. A rigorous meaning of these relations is given by formulas \eqref{utr67o}--\eqref{ufr7r}.

\item[(A5)] {\it Multiplication of operator-valued integrals:} For any
$$
h_1,\dots,h_{k+n}\in\mathscr H_{\mathbb C},\quad v^{(k)}\in L^\infty (\R^m,\nu^{(k)}_{\sharp_1,\dots,\sharp_k}),\quad w^{(n)}\in L^\infty (\R^n,\nu^{(n)}_{\sharp_{k+1},\dots,\sharp_{k+n}}),$$
we have
\begin{align}
&I^{(k)}_{\sharp_1,\dots,\sharp_k}(h_1,\dots,h_k,v^{(k)})I^{(n)}_{\sharp_{k+1},\dots,\sharp_{k+n}}(h_{k+1},\dots,h_{k+n},w^{(n)})\notag\\
&\quad=I^{(k+n)}_{\sharp_1,\dots,\sharp_{k+n}}(h_1,\dots,h_{k+n},v^{(k)}\otimes w^{(n)}).\notag
\end{align}
Here
\begin{align*}
&(v^{(k)}\otimes w^{(n)})(s_1,\dots,s_{k+n})\\
&\quad := v^{(k)}(s_1,\dots,s_k)w^{(n)}(s_{k+1},\dots,s_{k+n})\in L^\infty(\R^{k+n},\nu^{(k+n)}_{\sharp_1,\dots,\sharp_{k+n}}).\end{align*}

\end{itemize}

\begin{remark} We stress that the value $\eta$ of the function $Q$ on the diagonal does not matter for the relations \eqref{gt8}, \eqref{ugy8t78}.\end{remark}

Let $\mathbf A$ denote the complex algebra generated by the operator-valued integrals satisfying axioms (A1)--(A5), with the usual multiplication of operators acting on $\Theta$. We will call $\mathbf    A$ the {\it algebra of $Q$-commutation relations}, or the {\it $Q$-CR algebra} for short. In the case where the function $Q$ is given by
\eqref{tye6}, we will call $\mathbf    A$ the {\it algebra of anyon commutation relations}, or the {\it ACR algebra} for short.

The following theorem shows that the anyon exclusion principle \cite{GM} (see also \cite[Proposition~2.9]{BLW}) holds in
the ACR algebra  with $q$ being a root of 1.

\begin{theorem}\label{ftuf}
Let $k\in\mathbb N$, $k\ge 2$. Let $q\in\mathbb C$ be such that $q\ne1$ and $q^k=1$. Then, in the ACR algebra, we have,  for each $h\in\mathscr H_{\mathbb C}$,
$$\bigg(\int_X  m (dx)\, h(x)\partial_x^+\bigg)^k=0.$$
\end{theorem}

\begin{proof} Note that $\nu^{(k)}_{+,\dots,+}=\nu^{(k)}$, the Lebesgue measure on $\R^{(k)}$ ($\R^{(k)}$ consisting of all $(s_1,\dots,s_k)\in\R^k$ with $s_i\ne s_j$ if $i\ne j$).
By \eqref{gfr7},  we get, for any  $i\in\{1,\dots, k-1\}$ and $v^{(k)}\in L^\infty(\R^k,\nu^{(k)})$:
$$I_{+,\dots,+}^{(k)}(h,\dots,h,v^{(k)})=I_{+,\dots,+}^{(k)}(h,\dots,h,\Psi_i v^{(k)}).$$
Here $\Psi_i v^{(k)}$ is given by formula \eqref{utr67o}
 for $ (s_1,\dots,s_k)\in\R^{(k)}$.
By the proof of Proposition~2.8 in \cite{BLW}, it follows from here that, for each permutation $\pi\in S_k$, we get
 \begin{equation}
I_{+,\dots,+}^{(k)}(h,\dots,h,v^{(k)})=I_{+,\dots,+}^{(k)}(h,\dots,h,\Psi_\pi v^{(k)}),
 \label{gdyydfy}\end{equation}
 where
 $$ (\Psi_\pi v^{(k)})(s_1,\dots,s_k)=Q_{\pi^{-1}}(s_1,\dots,s_k)
 v^{(k)}(s_{\pi(1)},\dots,s_{\pi(k)}).$$
 Recall that the function $Q_\pi$ was defined by \eqref{higuy8a}, and we again used the obvious abuse of notation $Q_\pi(x_1,\dots,x_k)=Q_{\pi}(x_1^1,\dots,x_k^1)$ for $(x_1,\dots,x_k)\in X^k$.
 By \eqref{hgyi8} and \eqref{gdyydfy}, we  get
\begin{equation}
I_{+,\dots,+}^{(k)}(h,\dots,h,v^{(k)})=I_{+,\dots,+}^{(k)}(h,\dots,h,P_k v^{(k)}),
 \label{ufdty}\end{equation}
and the function $P_k v^{(k)}$ is $Q$-symmetric on $\R^{(k)}$. It follows from the proof of Proposition~2.9 in \cite{BLW} that if we choose $v^{(k)}\equiv 1$, we get
$$(P_k1)(s_1,\dots,s_k)= \frac{1-q^k}{(1-q)k!}=0$$
 for all $s_1<s_2<\dots<s_k$. Hence, $P_k1=0$ $\nu^{(k)}$-a.e. Now the statement of the theorem follows by the axioms (A5) and (A2).
\end{proof}

\subsection{Definition of a gauge-invariant quasi-free state}

Let $\tau$ be a state on the $Q$-CR algebra $\mathbf A$. Because of the $Q$-CR, $\tau$ is completely determined by its $n$-point functions, which are defined by formula \eqref{byut678} with $\varphi^{(k+n)}\in\mathbb F^{(k+n)}_{\sharp_1,\dots,\sharp_{k+n}}$, where
$\sharp_1=\dots=\sharp_k=+$ and $\sharp_{k+1}=\dots=\sharp_{k+n}=-$.
We already discussed in subsection \ref{vufyut} that gauge invariance of $\tau$ means that $\mathbf S^{(k,n)}\equiv 0$ if $k\ne n$. So our aim now is to introduce a proper generalization of formulas \eqref{cfyd6}, \eqref{iudgsy}.

We denote the $n$-point functions by
\begin{equation}\label{fu7}
\mathbf S^{(n,n)}(h_1,\dots,h_{2n},v^{(2n)}):=
\tau\big(I^{(2n)}_{\sharp_1,\dots,\sharp_{2n}}(h_1,\dots,h_{2n},v^{(2n)})\big),\end{equation}
where $\sharp_1=\dots=\sharp_n=+$ and $\sharp_{n+1}=\dots=\sharp_{2n}=-$.
By (A2), the right hand side of \eqref{fu7} identifies a linear functional of $v^{(2n)}\in L^\infty(\R^{2n},\nu^{(2n)}_{\sharp_1,\dots,\sharp_{2n}})$.
If we assume that this functional continuously depends on $v^{(2n)}$, then, according to the general theory of linear continuous functionals on $L^\infty$ spaces, this functional can be identified with a complex-valued, finite-additive measure on $\R^{2n}$ that is absolutely continuous with respect to the measure $\nu^{(2n)}_{\sharp_1,\dots,\sharp_{2n}}$. We will actually assume the following stronger condition to be satisfied.

\begin{itemize}
\item[(M)] For any $h_1,\dots,h_{2n}\in\mathscr H_{\mathbb C}$, there exists a (unique) complex-valued measure\linebreak  $\mathbf S^{(n,n)}[h_1,\dots,h_{2n}]$ on $\R^{2n}$ which is absolutely continuous with respect to $\nu^{(2n)}_{\sharp_1,\dots,\sharp_{2n}}$,  $\sharp_1=\dots=\sharp_n=+$ and $\sharp_{n+1}=\dots=\sharp_{2n}=-$, and satisfies, for all $v^{(2n)}\in L^\infty(\R^{2n},\nu^{(2n)}_{\sharp_1,\dots,\sharp_{2n}})$,
\begin{equation}\label{gufu7ro}
\mathbf S^{(n,n)}(h_1,\dots,h_{2n},v^{(2n)})=\int_{\R^{2n}}v^{(2n)}(s_1,\dots,s_{2n}) \,\mathbf S^{(n,n)}[h_1,\dots,h_{2n}](ds_1\dotsm ds_{2n}). \end{equation}
\end{itemize}

We will denote by $\mathbf S^{(n,n)}[h_1,\dots,h_{2n}](s_1,\dots, s_{2n})$ the density of the measure\linebreak $\mathbf S^{(n,n)}[h_1,\dots,h_{2n}](ds_1\dotsm ds_{2n})$ with respect to the measure
$\nu^{(2n)}_{\sharp_1,\dots,\sharp_{2n}}$ with   $\sharp_1=\dots=\sharp_n=+$ and $\sharp_{n+1}=\dots=\sharp_{2n}=-$.

We  say that $\tau$ is a {\it gauge-invariant quasi-free state on the $Q$-CR algebra $\mathbf A$} if $\mathbf S^{(k,n)}\equiv 0$ if $k\ne n$, and for each $n\in\mathbb N$ and any $g_1,\dots,g_n,h_1,\dots,h_n\in\mathscr H_{\mathbb C}$, we have
\begin{align}
&\mathbf S^{(n,n)}[g_n,\dots,g_1,h_1,\dots,h_n](s_n,\dots, s_1,s_{n+1},\dots, s_{2n})\notag\\
&\quad=\sum_{\pi\in S_n}\bigg(\prod_{i=1}^n\mathbf S^{(1,1)}[g_i,h_{\pi(i)}](s_i,s_{n+\pi(i)})\bigg) Q_\pi(s_1,\dots,s_n).\label{yu7r76}\end{align}

\begin{remark}Note the following slight difference in notations: in formulas \eqref{cfyd6}, \eqref{iudgsy}, the $n$-point function $\mathbf S^{(n,n)}(g_n,\dots,g_1,h_1,\dots,h_n)$ is linear in each $g_i$ and antilinear in each $h_i$, while in our setting the $n$-point function in \eqref{yu7r76} depends linearly on both $g_i$ and $h_i$.
\end{remark}

\begin{remark}Note that the measure $\nu^{(2n)}_{\sharp_1,\dots,\sharp_{2n}}$ with   $\sharp_1=\dots=\sharp_n=+$ and $\sharp_{n+1}=\dots=\sharp_{2n}=-$ remains invariant under the transformation
$$\R^{2n}\ni(s_1,\dots,s_{2n})\mapsto(s_n,\dots,s_1,s_{n+1},\dots,s_{2n})\in\R^{2n}.$$
 Hence, formulas \eqref{gufu7ro}, \eqref{yu7r76} mean that, for any $g_1,\dots,g_n,h_1,\dots,h_n\in\mathscr H_{\mathbb C}$ and
$v^{(2n)}\in L^\infty(\R^{2n},\nu^{(2n)}_{\sharp_1,\dots,\sharp_{2n}})$, we have
\begin{align}
&\mathbf S^{(n,n)}(g_n,\dots,g_1,h_1,\dots,h_n,v^{(2n)})= \int_{\R^{2n}}v^{(2n)}(s_n,\dots,s_1,s_{n+1},\dots, s_{2n})\notag\\
&\quad\times
\sum_{\pi\in S_n}\bigg(\prod_{i=1}^n\mathbf S^{(1,1)}[g_i,h_{\pi(i)}](s_i,s_{n+\pi(i)})\bigg) Q_\pi(s_1,\dots,s_n)\,\nu^{(2n)}_{\sharp_1,\dots,\sharp_{2n}}(ds_1\dotsm ds_{2n}).\label{utfr75re}
\end{align}

\end{remark}

Below, in Section \ref{uiy977}, we will explicitly construct a class
of gauge-invariant quasi-free states, but before doing this we wll now construct operator-valued integrals in the $Q$-Fock space.

\section{Operator-valued integrals in the $Q$-symmetric Fock space}\label{jhuf76}

In this section, we will assume that $X$ is a locally compact Polish space, $\mathscr B(X)$ is the Borel $\sigma$-algebra on $X$, and
$ m $ is a reference measure on $(X,\mathscr B(X))$. We assume $ m $ to be a Radon measure (i.e., finite on any compact set in $X$) and non-atomic. Analogously to subsection \ref{bit78}, we assume that $\Delta$ is a measurable, symmetric subset of $X^2$  and  satisfies $ m ^{\otimes 2}(\Delta)=0$.
We also assume that $D\subset\Delta$, where $D:=\{(x,x)\mid x\in X\}$ is the diagonal in $X^2$. We denote $\widetilde X^2:=X^2\setminus \Delta$. We fix $\eta\in\mathbb R$ and consider a function $Q: X^2\to\mathbb C$ such that  $|Q(x,y)|=1$ and $Q(x,y)=\overline{Q(y,x)}$ for all $(x,y)\in \widetilde X^2$, and $Q(x,y)=\eta$ for all $(x,y)\in\Delta$. 

Analogously to $\widetilde X^2$, we define, for each $n\ge3$
$$\widetilde X^n:=\big\{(x_1,\dots,x_n)\in X^n\mid (x_i,x_j)\not\in \Delta\text{ for all }1\le i<j\le n\big\}.$$
Let $n\ge 2$. A function $f^{(n)}:\widetilde X^n\to\mathbb C$ is called {\it $Q$-symmetric\/} if for any $i\in\{1,\dots,n-1\}$ and $(x_1,\dots,x_n)\in \widetilde X^n$, formula \eqref{ug8y87} holds.
Since $ m ^{\otimes n}(X^n\setminus\widetilde X^n)=0$, the function $f^{(n)}$ is defined $ m ^{\otimes n}$-a.e.\ on $X^n$. We define the function $Q_\pi:\widetilde X^n\to\mathbb C$ by formula \eqref{higuy8a}, and the $Q$-symmetrization of a function $f^{(n)}:\widetilde X^n\to\mathbb C$ by \eqref{hgyi8}.
The definitions of $\mathscr H$, $\mathscr H_{\mathbb C}$,  $\mathscr H_{\mathbb C}^{\circledast n}$, $P_n$, $\mathscr F^Q(\mathscr H)$, $\mathscr F_{\mathrm {fin}}^Q(\mathscr H)$, $a^+(h)$, and $a^-(h)$ are now similar to subsection~\ref{bit78}.

Let $\sharp_1,\dots,\sharp_k\in\{+,-\}$. Analogously to  \eqref{tye7i554o7}, \eqref{vugf7u}, we define
 a linear mapping
 $$\Xi^{(k)}_{\sharp_1,\dots,\sharp_k}: \mathscr H_{\mathbb C}^k\times L^\infty(X^k\to\mathbb C, m ^{(k)}_{\sharp_1,\dots,\sharp_k})\to L^0(X^k\to\mathbb C, m^{(k)}_{\sharp_1,\dots,\sharp_k})$$
 by
\begin{equation}\label{ftyr7i5}
\Xi^{(k)}_{\sharp_1,\dots,\sharp_k}\big[h_1,\dots,h_k,\varkappa^{(k)}\big] (x_1,\dots,x_k)
:=h_1(x_1)\dotsm h_k(x_k)\varkappa^{(k)}(x_1,\dots,x_k),
\end{equation}
where $h_1,\dots,h_k\in\mathscr H_{\mathbb C}$ and $\varkappa^{(k)}\in L^\infty(X^k\to\mathbb C, m ^{(k)}_{\sharp_1,\dots,\sharp_k})$.
We denote by $\mathbb G^{(k)}_{\sharp_1,\dots,\sharp_k}$ the
range of this mapping.  Our aim now is to construct operator-valued integrals of the form \eqref{yuur} with  $\varphi^{(k)}=\Xi^{(k)}_{\sharp_1,\dots,\sharp_k}\big[h_1,\dots,h_k,\varkappa^{(k)}\big]$.

The following proposition follows immediately from the definition of the creation operator, $a^+(h)$, and \cite[Proposition~3.2]{BLW}.

\begin{proposition}\label{urfte5}
Let $\mathscr F(\mathscr H):=\bigoplus_{n=0}^\infty \mathscr H_{\mathbb C}^{\otimes n}$ denote the {\it full Fock space over $\mathscr H$}, and let the space $\mathscr F_{\mathrm{fin}}(\mathscr H)$ be defined analogously to $\mathscr F_{\mathrm{fin}}^Q(\mathscr H)$.
For $h\in\mathscr H_{\mathbb C}$, we define  linear continuous operator  $b^+(h)$and $b^-(h)$ on $\mathscr F_{\mathrm{fin}}(\mathscr H)$ by
\begin{align}
&b^+(h)f^{(n)}:=h\otimes f^{(n)},\notag\\
&(b^-(h)f^{(n)})(x_1,\dots,x_{n-1}):=\sum_{i=1}^n\int_X h(y)
Q(y,x_1)\dotsm Q(y,x_{i-1})\notag\\
&\quad \times f^{(n)}(x_1,\dots,x_{i-1},y,x_i,\dots,x_{n-1})\, m (dy),\quad f^{(n)}\in\mathscr H_{\mathbb C}^{\otimes n}\label{utfrt7}
\end{align}
Then, on  $\mathscr F_{\mathrm{fin}}(\mathscr H)$, we have
\begin{equation}\label{uyr678}
 a^+(h)P=Pb^+(h),\quad a^-(h)P=Pb^-\big(\bar h\big).\end{equation}
  Here, for $f^{(n)}\in\mathscr H_{\mathbb C}^{\otimes n}$, we define $Pf^{(n)}:=P_nf^{(n)}$.
\end{proposition}

Using the notation \eqref{furt7i},  we conclude from here the following corollary.

\begin{corollary}\label{cyre65}
For any $\sharp_1,\dots,\sharp_k\in\{+,-\}$ and any $h_1,\dots,h_k\in\mathscr H_{\mathbb C}$, we have on $\mathscr F_{\mathrm{fin}}^Q(\mathscr H)$:
\begin{equation}\label{iyuf7r}
\int_{X^k} m ^{\otimes k}(dx_1\dotsm dx_k)\, h_1(x_1)\dotsm h_k(x_k)\partial_{x_1}^{\sharp_1}\dotsm \partial_{x_k}^{\sharp_k}
=Pb^{\sharp_1}(h_1)\dotsm b^{\sharp_k}(h_k).\end{equation}
Here we denoted
\begin{align*}
&\int_{X^k} m ^{\otimes k}(dx_1\dotsm dx_k) h_1(x_1)\dotsm h_k(x_k)\partial_{x_1}^{\sharp_1}\dots, \partial_{x_k}^{\sharp_k}\\
&\quad:=\bigg(\int_X m (dx_1)h(x_1)\partial_{x_1}^{\sharp_1}\bigg)\dotsm \bigg(\int_X m (dx_k)h(x_k)\partial_{x_k}^{\sharp_k}\bigg). \end{align*}
\end{corollary}

Let $f^{(n)}\in\mathscr H_{\mathbb C}^{\circledast n}$. Using Corollary~\ref{cyre65}, we can write down the action of the operator
in formula  \eqref{iyuf7r} on $f^{(n)}$ through the function
\begin{equation}\label{gyd65rir}
 \varphi^{(k)}(x_1,\dots,x_k):= h_1(x_1)\dotsm h_k(x_k).\end{equation}
(Note that this function is defined $ m ^{(k)}_{\sharp_1,\dots,\sharp_k}$-a.e.) For example,
\begin{align}
&\bigg(\int_{X^3} m ^{\otimes 3}(dy_1\,dy_2\,dy_3)\varphi^{(3)}(y_1,y_2,y_3)\partial_{y_1}^-\partial_{y_2}^-\partial_{y_3}^+ f^{(n)}\bigg)(x_1,\dots,x_{n-1})\notag\\
& =P_{n-1}\bigg\{\sum_{i=1}^n\int_{X^2} m ^{\otimes 2}(dy_1\,dy_2)\,
\varphi^{(3)}(y_1,y_2,y_2) \notag\\
&\quad\times
\big[Q(y_1,x_1)\dotsm Q(y_1,x_{i-1})
f^{(n)}(x_1,\dots,x_{i-1},y_1,x_i,\dots,x_{n-1})\notag\\
&\qquad+Q(y_2,y_1)Q(y_2,x_1)\dotsm Q(y_2,x_{i-1})f^{(n)}(x_1,\dots,x_{i-1},y_2,x_i,\dots,x_{n-1})\big]
\notag\\
& +\sum_{i=2}^{n+1}\int_{X^2} m ^{\otimes 2}(dy_1\,dy_2)\,
\varphi^{(3)}(y_1,y_2,x_1)\bigg[\sum_{j=2}^{i-1}
Q(y_2,y_1)
Q(y_1,x_1)\dotsm Q(y_1,x_{j-1})\notag\\
&\quad\times Q(y_2,x_1)\dotsm Q(y_2,x_{i-2})f^{(n)}(x_2,\dots,x_{j-1},y_1,x_j,\dots,x_{i-2},y_2,x_{i-1},\dots,x_{n-1})
 \notag\\
 & +\sum_{j=i}^{n}Q(y_1,x_1)\dotsm Q(y_1,x_{j-1})Q(y_2,x_1)\dotsm Q(y_2,x_{i-1})\notag\\
 &\quad\times f^{(n)}(x_2,\dots,x_{i-1},y_2,x_i,\dots,x_{j-1},y_1,x_j,\dots,x_{n-1})\bigg]\bigg\}.\notag
\end{align}

As easily seen, we can replace in the obtained formulas the function $\varphi^{(k)}$ of the form \eqref{gyd65rir} with a function $\varphi^{(k)}$ being given by the right hand side of formula \eqref{ftyr7i5}. As a result, for each $\varphi^{(k)}=\Xi^{(k)}_{\sharp_1,\dots,\sharp_k}\big[h_1,\dots,h_k,\varkappa^{(k)}\big]\in \mathbb G^{(k)}_{\sharp_1,\dots,\sharp_k}$, we have constructed a continuous linear operator on  $\mathscr F^Q_{\mathrm{fin}}(\mathscr H)$ which is denoted as in \eqref{yuur}  or by $I^{(k)}_{\sharp_1,\dots,\sharp_k}(h_1,\dots,h_k,\varkappa^{(k)})$. Note that this operator indeed depends  on the values of the function $\varphi^{(k)}(x_1,\dots,x_k)=h_1(x_1)\dotsm h_k(x_k)\varkappa^{(k)}(x_1,\dots,x_k)$ $ m ^{(k)}_{\sharp_1,\dots,\sharp_k}$-a.e.

The following proposition follows from the construction of the operator-valued integrals (compare with \cite[Proposition~3.8]{BLW} regarding the corresponding statement about the $Q$-CR).

\begin{proposition}\label{uf7fi}
The above constructed operator-valued integrals $I^{(k)}_{\sharp_1,\dots,\sharp_k}(h_1,\dots,h_k,\varkappa^{(k)})$ satisfy the axioms {\rm (A1)${}'$--(A5)${}'$} that are obtained from the axioms {\rm (A1)--(A5)} by replacing the space $L^\infty(\R^k\to\mathbb C,\nu^{(k)}_{\sharp_1,\dots,\sharp_k})$  by $L^\infty(X^k\to\mathbb C, m ^{(k)}_{\sharp_1,\dots,\sharp_k})$. \end{proposition}

Let us note that we initially had operator-valued integrals $I^{(k)}_{\sharp_1,\dots,\sharp_k}(h_1,\dots,h_k,\varkappa^{(k)})$ for $\varkappa^{(k)}\equiv 1$ and then we extended them to an arbitrary function $\varkappa^{(k)}\in L^\infty(X^k\to\mathbb C, m ^{(k)}_{\sharp_1,\dots,\sharp_k})$.
The following lemma shows that, under  natural assumptions on the operator-valued integrals, this extension is, in fact, unique.

\begin{lemma}\label{tuyfr7}
Assume that the operator-valued integrals $\tilde I^{(k)}_{\sharp_1,\dots,\sharp_k}(h_1,\dots,h_k,\varkappa^{(k)})$ acting on $\mathscr F^Q_{\mathrm{fin}}(\mathscr H)$ satisfy the axioms {\rm (A1)${}'$, (A2)${}'$} and the following assumption: for any
$h_1,\dots,h_k\in\mathscr H_{\mathbb C}$ and any
$F,G\in \mathscr F^Q_{\mathrm{fin}}(\mathscr H)$ the linear functional
$$L^\infty(X^k\to\mathbb C, m ^{(k)}_{\sharp_1,\dots,\sharp_k})\ni\varkappa^{(k)}\to \big(\tilde I^{(k)}_{\sharp_1,\dots,\sharp_k}(h_1,\dots,h_k,
\varkappa^{(k)})F,G\big)_{\mathscr F^Q(\mathscr H)} \in\mathbb C$$
 is continuous and is given by a complex-valued measure on $X^k$. Then the equality
\begin{equation}\label{vtyr7i}
\tilde I^{(k)}_{\sharp_1,\dots,\sharp_k}(h_1,\dots,h_k,
1)=I^{(k)}_{\sharp_1,\dots,\sharp_k}(h_1,\dots,h_k,
1) \end{equation}
implies the equality
$$\tilde I^{(k)}_{\sharp_1,\dots,\sharp_k}(h_1,\dots,h_k,
\varkappa^{(k)})=I^{(k)}_{\sharp_1,\dots,\sharp_k}(h_1,\dots,h_k,
\varkappa^{(k)}) $$
for all $\varkappa^{(k)}\in L^\infty(X^k\to\mathbb C, m ^{(k)}_{\sharp_1,\dots,\sharp_k})$.
\end{lemma}

\begin{proof} Denote by $\tilde{\mathfrak m}_{\sharp_1,\dots,\sharp_k}^{(k)}[F;G;h_1,\dots,h_k]$ the complex-valued measure on $X^k$ that satisfies, for all $\varkappa^{(k)}\in L^\infty(X^k\to\mathbb C, m ^{(k)}_{\sharp_1,\dots,\sharp_k})$,
\begin{align}
&\int_{X^k}\varkappa^{(k)}(x_1,\dots,x_k)\,\tilde{\mathfrak m}_{\sharp_1,\dots,\sharp_k}^{(k)}[F;G;h_1,\dots,h_k](dx_1\dotsm dx_k)\notag\\
&\quad =\big(\tilde I^{(k)}_{\sharp_1,\dots,\sharp_k}(h_1,\dots,h_k,
\varkappa^{(k)})F,G\big)_{\mathscr F^Q(\mathscr H)}.\label{tuyr7}
\end{align}
As easily seen from the construction of the operator-valued integrals,  there  exists a complex-valued measure  ${\mathfrak m}_{\sharp_1,\dots,\sharp_k}^{(k)}[F;G;h_1,\dots,h_k]$ on $X^k$ that satisfies
equality \eqref{tuyr7} in which $\tilde{\mathfrak m}_{\sharp_1,\dots,\sharp_k}^{(k)}$ and $\tilde I^{(k)}_{\sharp_1,\dots,\sharp_k}$ are replaced with ${\mathfrak m}_{\sharp_1,\dots,\sharp_k}^{(k)}$ and $ I^{(k)}_{\sharp_1,\dots,\sharp_k}$, respectively. Hence, it suffices to prove that
\begin{equation}\label{gufk7}
\tilde{\mathfrak m}_{\sharp_1,\dots,\sharp_k}^{(k)}[F;G;h_1,\dots,h_k]={\mathfrak m}_{\sharp_1,\dots,\sharp_k}^{(k)}[F;G;h_1,\dots,h_k].\end{equation}
For $i\in\{1,\dots,k\}$, let $A_i\in\mathscr B(X)$ and denote by $\chi_{A_i}$ the indicator function of the set $A_i$.
We have, by axiom (A1)${}'$ and \eqref{vtyr7i},
\begin{align}
&\tilde{\mathfrak m}_{\sharp_1,\dots,\sharp_k}^{(k)}[F;G;h_1,\dots,h_k](A_1\times\dots\times A_k)
 =\big(\tilde I^{(k)}_{\sharp_1,\dots,\sharp_k}(h_1,\dots,h_k,
\chi_{A_1}\dotsm \chi_{A_k})F,G\big)_{\mathscr F^Q(\mathscr H)}\notag\\
&= \big(\tilde I^{(k)}_{\sharp_1,\dots,\sharp_k}(h_1\chi_{A_1},\dots,h_k\chi_{A_k},
1)F,G\big)_{\mathscr F^Q(\mathscr H)}=\big(I^{(k)}_{\sharp_1,\dots,\sharp_k}(h_1\chi_{A_1},\dots,h_k\chi_{A_k},
1)F,G\big)_{\mathscr F^Q(\mathscr H)}\notag\\
& ={\mathfrak m}_{\sharp_1,\dots,\sharp_k}^{(k)}[F;G;h_1,\dots,h_k](A_1\times\dots\times A_k),\notag
\end{align}
which implies \eqref{gufk7}.
  \end{proof}

 We note that, in this section, we have not yet used the value of the function $Q$ on the diagonal.

 \begin{proposition}\label{ur7ir}
 The following relation between operators $\partial_x^+$ and $\partial_y^-$ holds in the obtained representation of the $Q$-CR algebra  in the $Q$-Fock space  $\mathscr F^Q(\mathscr H)$:
  \begin{equation}\label{tyr65i}
\partial_x^+\partial^-_y=Q(x,y)\partial^-_y\partial_x^+-\eta\delta(x,y).\end{equation}
Here $\int_{X^2} f^{(2)}(x,y)\delta(x,y)\, m ^{\otimes 2}(dx\,dy):=\int_X f^{(2)}(x,x)\, m (dx)$.
A rigorous meaning of relation \eqref{tyr65i} is given analogously to formulas \eqref{erererertyr75}--\eqref{ufr7r} (see also formulas \eqref{yur76o}, \eqref{yre6e} below).
 \end{proposition}

 \begin{proof}
 To simplify notation, let us consider the case of an operator-valued integral $I^{(2)}_{+,-}(h_1,h_2,\varkappa^{(2)})$. Since $X^2$ can be represented as the disjoint union of $\widetilde X^2$ and $\Delta$, we have
 $$I^{(2)}_{+,-}(h_1,h_2,\varkappa^{(2)})=I^{(2)}_{+,-}(h_1,h_2,\varkappa^{(2)}\chi_{\widetilde X^2})+I^{(2)}_{+,-}(h_1,h_2,\varkappa^{(2)}\chi_{\Delta}).$$
But $I^{(2)}_{+,-}(h_1,h_2,\varkappa^{(2)}\chi_{\Delta})=I^{(2)}_{+,-}(h_1,h_2,0)=0$.
 Hence, in view of the relation
 $$\partial_x^-\partial_y^+=\delta(x,y)+Q(x,y)\partial_y^+\partial_x^-,$$ we get
 \begin{equation}\label{ytr7i5}
 I^{(2)}_{+,-}(h_1,h_2,\varkappa^{(2)})=I^{(2)}_{+,-}(h_1,h_2,\varkappa^{(2)}\chi_{\widetilde X^2})=
 I^{(2)}_{-,+}(h_2,h_1, (\Psi_1'\varkappa^{(2)})\chi_{\widetilde X^2}),\end{equation}
 where
 \begin{equation}\label{yur76o}
 \Psi_1'\varkappa^{(2)}(x_1,x_2)=\varkappa^{(2)}(x_2,x_1)Q(x_2,x_1).\end{equation}
 On the other hand
 \begin{align}&I^{(2)}_{-,+}(h_2,h_1, \Psi_1'\varkappa^{(2)})\notag\\
&\quad=I^{(2)}_{-,+}(h_2,h_1, (\Psi_1'\varkappa^{(2)})\chi_{\widetilde X^2})+ \int_X h_1(x)h_2(x)\varkappa^{(2)}(x,x)Q(x,x)\, m (dx)\notag\\
 &\quad=I^{(2)}_{-,+}(h_2,h_1, (\Psi_1'\varkappa^{(2)})\chi_{\widetilde X^2})+ \eta\int_X h_1(x)h_2(x)\varkappa^{(2)}(x,x)\, m (dx).\label{uftr75}\end{align}
 By \eqref{ytr7i5} and \eqref{uftr75},
\begin{equation}\label{yre6e}
I^{(2)}_{+,-}(h_1,h_2,\varkappa^{(2)})=I^{(2)}_{-,+}(h_2,h_1, \Psi_1'\varkappa^{(2)})-\eta\int_X h_1(x)h_2(x)\varkappa^{(2)}(x,x)\, m (dx).\end{equation}
 \end{proof}

\begin{remark}
Assume that $X=\R^d$ with $d\ge2$, $ m $ is the Lebesgue measure on $X$ and $Q(x,y)=Q(x^1,y^1)$. Then, we have the inclusion  $\mathbb F^{(k)}_{\sharp_1,\dots,\sharp_k}\subset \mathbb G^{(k)}_{\sharp_1,\dots,\sharp_k}$, where we identify each function $v^{(k)}(s_1,\dots,s_k)\in L^\infty(\R^k\to\mathbb C,\nu^{(k)}_{\sharp_1,\dots,\sharp_k})$ with the function
$$\varkappa^{(k)}(x_1,\dots,x_k):=v^{(k)}(x_1^1,\dots,x_k^1)\in L^\infty(X^k\to\mathbb C, m ^{(k)}_{\sharp_1,\dots,\sharp_k}).$$
Thus, the above constructed operator-valued integrals give a representation of the $Q$-CR algebra in the Fock space $\mathscr F^Q(\mathscr H)$.
\end{remark}

\section{Construction of gauge-invariant quasi-free states}\label{uiy977}

In this section we again assume  $X=\R^d$ with $d\ge2$,   $ m $ is the Lebesgue measure,   $Q(x,y)=Q(x^1,y^1)$ for $(x,y)\in\widetilde X^2$  and $Q(x,y)=\eta\in\mathbb R$ for $(x,y)\in\Delta$.

\subsection{The operators $K_1,K_2$}\label{gut8}

We fix continuous linear operators $K_1$ and $K_2$ in $\mathscr H$. We assume that these operators satisfy the following condition.

\begin{itemize}
\item[(C)] For a bounded measurable function $\psi:\R\to\R$, let $M_\psi$ denote the continuous linear operator in $\mathscr H$ given by
$$(M_\psi f)(x^1,\dots,x^d):=\psi(x^1)f(x^1,\dots,x^d), \quad f\in \mathscr H.$$
Then, for any bounded measurable function $\psi:\R\to\R$, both operators $K_1$ and $K_2$ commute with $M_\psi$.
\end{itemize}

\begin{remark}\label{yf7}
Condition (C) implies that,  for any bounded measurable function $\psi:\R\to\R$, both operators $K^*_1$ and $K^*_2$ commute with $M_\psi$.
\end{remark}

\begin{remark}\label{uyr75khuyh}
Condition (C) is satisfied if $K_i=\mathbf 1 \otimes \widetilde K_i$ ($i=1,2$), where $\widetilde K_1$  and $\widetilde K_2$ are any continuous linear operators in $L^2(\R^{d-1},dx^{2}\dotsm dx^{d})$. In the general case,  the operators $K_i$ have the following structure:
$$ (K_if)(x^1,x^2,\dots,x^d)=\big(\widetilde K_i (x^1)f(x^1,\cdot)\big)(x^2,\dots,x^d),\quad i=1,2,$$
where, for each $x^1\in\R$, $K_i(x^1)$ is a continuous linear operator in $L^2(\R^{d-1},dx^{2}\dotsm dx^{d})$ such that $K_i$ is a continuous linear operator in $\mathscr H$.
\end{remark}

\begin{remark}
The results of this section with a proper modification will also hold for $X=\R$. In this case, condition (C) just means that both $K_1$ and $K_2$ are multiplication operators. In fact, under the latter assumption, we could deal with an arbitrary locally compact Polish space $X$ and a function $Q $
as in Section \ref{jhuf76}.
\end{remark}

We extend the operators $K_1$ and $K_2$ by linearity to $\mathscr H_{\mathbb C}$, the complexification of $\mathscr H$.

\begin{remark}\label{yut86}
 Note that, for each $h\in\mathscr H_{\mathbb C}$, we get $\overline{K_ih}=K_i\bar h$ and similarly for $K_i^*$, $i=1,2$.\end{remark}

Let $i_1,\dots,i_k\in\{1,2\}$ and let  us consider the operator $K_{i_1}\otimes\dots\otimes K_{i_k}$ in $\mathscr H_{\mathbb C}^{\otimes k}$. Let $h_1,\dots,h_k\in\mathscr H_{\mathbb C}$ and let $v^{(k)}\in L^\infty(\R^k,\nu^{(k)})$. Then
$$\varphi^{(k)}(x_1,\dots,x_k):=h_1(x_1)\dotsm h_k(x_k)v^{(k)}(x^1_1,\dots,x^1_k)$$
belongs to $\mathscr H_{\mathbb C}^{\otimes k}.$
By using condition (C), we get the following equality in $\mathscr H_{\mathbb C}^{\otimes k}$, hence $ m ^{\otimes k}$-a.e.:
$$ (K_{i_1}\otimes\dots\otimes K_{i_k})\varphi^{(k)}(x_1,\dots,x_k)=(K_{i_1}h_1)(x_1)\dotsm (K_{i_k}h_k)(x^k) v^{(k)}(x^1_1,\dots,x^1_k).$$
Hence, we may define a linear operator
$$K_{i_1}\otimes\dots\otimes K_{i_k}: \mathbb F^{(k)}_{\sharp_1,\dots,\sharp_k}\to \mathbb F^{(k)}_{\sharp_1,\dots,\sharp_k}$$
by
\begin{equation}\label{uyr7rvgf}
(K_{i_1}\otimes\dots\otimes K_{i_k})\Phi _{\sharp_1,\dots,\sharp_k}^{(k)}\big[h_1,\dots,h_k,v^{(k)}\big]:= \Phi _{\sharp_1,\dots,\sharp_k}^{(k)}\big[K_{i_1}h_1,\dots,K_{i_k}h_k,v^{(k)}\big].\end{equation}
Indeed, the action of $K_{i_1}\otimes\dots\otimes K_{i_k}$  onto $\varphi^{(k)}\in \mathbb F^{(k)}_{\sharp_1,\dots,\sharp_k}$ is independent of the representation \eqref{dyftyd}.

\subsection{The representation of the $Q$-CR algebra corresponding to the operators $K_1,K_2$}\label{ut0u0}

Given operators $K_1$, $K_2$ satisfying condition (C), we will now construct a corresponding representation of the $Q$-CR algebra. Our construction is reminiscent of construction  of quasi-free states for the CCR and CAR cases using the representations of Araki, Woods \cite{AWoods} and Araki, Wyss \cite{AWyss}, respectively.

Let $X_1$ and $X_2$ denote two copies of the space $X$.
Let $Z:=X_1\sqcup X_2$ denote the disjoint union of $X_1$ and $X_2$.  Thus, $Z=X\times\{1,2\}$. We equip $Z$ with the product topology of the space $X$ and the trivial one on $\{1,2\}$. 
In particular,  $Z$ is a locally compact Polish space. With an abuse of notation, we define a measure $ m $ on $(Z,\mathscr B(Z))$ so that the restriction of this measure to $X_1$ (or $X_2$, respectively) coincides with the measure $ m $ on $(X,\mathscr B(X))$. In particular, we get
$$L^2(Z\to\mathbb C, m )=L^2(X_1\to\mathbb C, m )\oplus L^2(X_2\to\mathbb C, m )=\mathscr H_{\mathbb C}\oplus\mathscr H_{\mathbb C}.$$

On some occasions, we will identify a point $(x,y)\in Z^2$ with the corresponding point $(x,y)\in X^2$, i.e., we forget  which of the two copies of the space $X$ the points $x$ and $y$ belong to. So, again  with an abuse of notation, we define a subset $\Delta$ of $Z^2$ which consists of those points $(x,y)\in Z^2$ for which $(x,y)\in\Delta$, where the latter $\Delta$ is the above introduced subset of $X^2$. Note that $ m ^{\otimes 2}(\Delta)=0$ and $\Delta$ contains the diagonal in $Z^2$.
 Similarly, if
$\phi:X^2\to\mathbb C$ is a function on $X^2$ and if $(x,y)\in Z^2$, we will denote by $\phi(x,y)$ the value of the function $\phi$ at the corresponding point $(x,y)\in X^2$.

Let a function $JQ:Z^2\to\mathbb C$ be defined by
\begin{equation}\label{oyyoy}
 JQ(x,y):=\begin{cases}
Q(x,y),&\text{if $x,y\in X_1$ or $x,y\in X_2$},\\
Q(y,x),&\text{if $x\in X_1$, $y\in X_2$ or $x\in X_2$, $y\in X_1$}.
\end{cases}\end{equation}
In particular, $JQ(x,y)=\eta$ for all $(x,y)\in\Delta$ and $|JQ(x,y)|=1$, $JQ(y,x)=\overline{JQ(x,y)}$ for all $(x,y)\in \widetilde Z^2:=Z^2\setminus \Delta$.

So, according to Section \ref{jhuf76}, we can define the $JQ$-Fock space over $L^2(Z, m )$, i.e., $\mathscr F^{JQ}(L^2(Z, m ))$.
For $x\in X$, we denote by $\partial_{x,i}^+$ and $\partial_{x,i}^-$ ($i\in\{1,2\}$) the  creation and annihilation operators at the point $x$ being identified with the corresponding point of $X_i$. Thus, analogously to \eqref{furt7i}, we may write, for $h\in\mathscr H_{\mathbb C}$,
\begin{align}
&a^+(h,0)=\int_X m (dx)\,h(x)\partial_{x,1}^+, &a^-(h,0)=\int_X m (dx)\,\overline{h(x)}\,\partial_{x,1}^-,\notag\\
&a^+(0,h)=\int_X m (dx)\,h(x)\partial_{x,2}^+, &a^-(0,h)=\int_X m (dx)\,\overline{h(x)}\,\partial_{x,2}^-. \label{uyr7r7}
\end{align}

We now define (informal) operators $D_x^+$ and $D_x^-$ ($x\in X$) which satisfy, for each $h\in\mathscr H_{\mathbb C}$:
\begin{align}
&\int_X m (dx)\,h(x)D_x^+:=\int_X  m (dx)\,(K_1h)(x)\partial_{x,1}^-+\int_X m (dx)\,(K_2h)(x)\partial_{x,2}^+,\label{vgydy6d6}\\
&\int_X m (dx)\,h(x)D_x^-:=\int_X  m (dx)\,(K_1h)(x)\partial_{x,1}^++\int_X m (dx)\,(K_2h)(x)\partial_{x,2}^-.\label{guf7}
\end{align}
We will now show that, under the assumption \eqref{iyt8o} below, the operators $D_x^+$,  $D_x^-$ satisfy the $Q$-CR and lead to a representation of the $Q$-CR algebra. The latter algebra will be generated by the operator-valued integrals
\begin{align}
&\mathfrak I^{(k)}_{\sharp_1,\dots,\sharp_k}(h_1,\dots,h_k,v^{(k)})\notag\\
&\quad:=
\int_{X^k} m ^{\otimes k}(dx_1\dotsm dx_k)\,h_1(x_1)\dotsm h_k(x_k)v^{(k)}(x_1^1,\dots,x_k^1) D_{x_1}^{\sharp_1}\dotsm D_{x_k}^{\sharp_k},\label{vuf7r}\end{align}
where $\sharp_1,\dots,\sharp_k\in\{+,-\}$, $h_1,\dots,h_k\in\mathscr H_{\mathbb C}$, and $v^{(k)}\in L^\infty(\R^k,\nu^{(k)}_{\sharp_1,\dots,\sharp_k})$.
In view of \eqref{uyr7rvgf}--\eqref{guf7}, we can  now easily formalize the definition \eqref{vuf7r}.

We define operators $\mathscr K_1:\mathscr H_{\mathbb C}\to \mathscr H_{\mathbb C}\oplus \mathscr H_{\mathbb C}$ by
\begin{equation}\label{uyr75yfttr}
 \mathscr K_1h:=(K_1h,0),\quad \mathscr K_2 h:=(0,K_2h),\quad h\in\mathscr H_{\mathbb C}.\end{equation}
We also denote
\begin{equation}\label{yufu8rt86rd}
 s(1,+):=-,\quad s(2,+):=+,\quad s(1,-):=+,\quad s(2,-):=-.\end{equation}
Using these notations, we can rewrite formulas \eqref{vgydy6d6}, \eqref{guf7} as follows:
\begin{align}
&\int_X m (dx)\,h(x)D_x^+=\int_Z  m (dx)\,(\mathscr K_1h)(x)\partial_{x}^{s(1,+)}+\int_Z m (dx)\,(\mathscr K_2h)(x)\partial_{z}^{s(2,+)},\notag\\
&\int_X m (dx)\,h(x)D_x^-=\int_Z  m (dx)\,(\mathscr K_1h)(x)\partial_{z}^{s(1,-)}+\int_Z m (dx)\,(\mathscr K_2h)(x)\partial_{x}^{s(2,-)}.\notag
\end{align}

For $g_1,\dots,g_k\in\mathscr H_{\mathbb C}\oplus\mathscr H_{\mathbb C}$ and $\varkappa^{(k)}\in L^\infty (Z^k\to\mathbb C, m ^{(k)}_{\sharp_1,\dots,\sharp_k})$, we denote by
\begin{equation}\label{hiyt8o}I^{(k)}_{\sharp_1,\dots,\sharp_k}(g_1,\dots,g_k,\varkappa^{(k)})\end{equation}
the corresponding operator-valued integral in the $JQ$-Fock space $\mathscr F^{JQ}(L^2(Z, m ))$ acting on $\mathscr F_{\mathrm{fin}}^{JQ}(L^2(Z, m ))$ as defined in Section~\ref{jhuf76}.

We now give a rigorous formulation of the definition \eqref{vuf7r}.
For any $\sharp_1,\dots,\sharp_k\in\{+,-\}$, $h_1,\dots,h_k\in\mathscr H_{\mathbb C}$, and $v^{(k)}\in L^\infty(\R^k,\nu^{(k)}_{\sharp_1,\dots,\sharp_k})$, we define
\begin{align}
&\mathfrak I^{(k)}_{\sharp_1,\dots,\sharp_k}(h_1,\dots,h_k,v^{(k)})\notag\\
&\quad:=\sum_{(i_1,\dots,i_k)\in\{1,2\}^k}I^{(k)}_{s(i_1,\sharp_1),\dots ,s(i_k,\sharp_k)}(\mathscr K_{i_1}h_1,\dots,\mathscr K_{i_k}h_k,\mathscr R^{(k)}_{i_1,\dots,i_k}v^{(k)}),\label{iyut8}
\end{align}
where the function $\mathscr R^{(k)}_{i_1,\dots,i_k}v^{(k)}\in L^\infty (Z^k\to\mathbb C, m ^{(k)}_{\sharp_1,\dots,\sharp_k})$ is given by
\begin{equation}\label{buyt8t}
\big(\mathscr R^{(k)}_{i_1,\dots,i_k}v^{(k)}\big)(x_1,\dots,x_k):=\begin{cases}
v^{(k)}(x_1^1,\dots,x_k^1),&\text{if }(x_1,\dots,x_k)\in X_{i_1}\times\dots\times X_{i_k},\\
0,&\text{otherwise}.
\end{cases}\end{equation}

\begin{theorem}\label{tre645}
Let $K_1$ and $K_2$ be continuous linear operators in $\mathscr H$ which satisfy condition (C) and
\begin{equation}\label{iyt8o}
K_2^*K_2=\mathbf 1+\eta K_1^*K_1.
\end{equation}
Let the function $JQ:Z\to\mathbb C$ be defined by \eqref{oyyoy}. Let for any $\sharp_1,\dots,\sharp_k\in\{+,-\}$, $h_1,\dots,h_k\in\mathscr H_{\mathbb C}$, and $v^{(k)}\in L^\infty(\R^k,\nu^{(k)}_{\sharp_1,\dots,\sharp_k})$, a continuous linear operator
$$\mathfrak I^{(k)}_{\sharp_1,\dots,\sharp_k}(h_1,\dots,h_k,v^{(k)})$$
on $\mathscr F_{\mathrm{fin}}^{JQ}(L^2(Z, m ))$ be defined by \eqref{iyut8}. These operators satisfy the axioms {\rm (A1)--(A5)}. Thus, the algebra $\mathbf A$ generated by these operators gives a representation of the $Q$-CR algebra.
\end{theorem}

\begin{proof}
Axiom (A1) follows from condition (C) and the considerations in the last paragraph of subsection~\ref{gut8}. Axiom (A2) is trivially satisfied. Axiom (A3) follows from the corresponding property
of the operator-valued integrals \eqref{hiyt8o} and and Remark~\ref{yut86}. Similarly, axiom (A5) is trivially satisfied. So we only have to check the axiom (A4), i.e., the $Q$-CR relations.

We will  prove (A4) only in the case $k=2$. The general case will follow from the case $k=2$ by an easy generalization, which we leave to the interested reader.

So let $h_1,h_2\in \mathscr H_{\mathbb C}$ and $v^{(2)}\in L^\infty(\R^2\to\mathbb C,\nu^{(2)}_{+,+})$, we get from \eqref{oyyoy}, \eqref{uyr75yfttr}, \eqref{yufu8rt86rd}, \eqref{iyut8}, \eqref{buyt8t}, Proposition \ref{ur7ir}, and the $JQ$-CR satisfied by the operator-valued integrals \eqref{hiyt8o}:
\begin{align*}
&\mathfrak I^{(2)}_{+,+}(h_1,h_2,v^{(2)})=
I^{(2)}_{-,-}(\mathscr K_1 h_1,\mathscr K_1h_2,\mathscr R_{1,1}^{(2)}v^{(2)})+
I^{(2)}_{+,+}(\mathscr K_2h_1,\mathscr K_2h_2,\mathscr R_{2,2}^{(2)}v^{(2)})\\
&\qquad+I^{(2)}_{-,+}(\mathscr K_1h_1,\mathscr K_2h_2,\mathscr R^{(2)}_{1,2}v^{(2)})+
I^{(2)}_{+,-}(\mathscr K_2h_1,\mathscr K_1h_2,\mathscr R^{(2)}_{2,1}v^{(2)})\\
&\quad=
I^{(2)}_{-,-}(\mathscr K_1 h_2,\mathscr K_1h_1,\mathscr R_{1,1}^{(2)}\Psi_1v^{(2)})+
I^{(2)}_{+,+}(\mathscr K_2h_2,\mathscr K_2h_1,\mathscr R_{2,2}^{(2)}\Psi_1v^{(2)})\\
&\qquad+I^{(2)}_{+,-}(\mathscr K_2h_2,\mathscr K_1h_1,\mathscr R^{(2)}_{2,1}\Psi_1v^{(2)})+
I^{(2)}_{-,+}(\mathscr K_1h_2,\mathscr K_2h_1,\mathscr R^{(2)}_{1,2}\Psi_1v^{(2)})\\
&\quad=\mathfrak I^{(2)}_{+,+}(h_2,h_1,\Psi_1v^{(2)}).
\end{align*}
By taking the adjoint operators, this also implies that
$$ \mathfrak I^{(2)}_{-,-}(h_1,h_2,v^{(2)})=\mathfrak I^{(2)}_{+,+}(h_2,h_1,\Psi_1v^{(2)}).$$
Using additionally  Remark \ref{yf7}  and \eqref{iyt8o}, we get
\begin{align*}
&\mathfrak I^{(2)}_{-,+}(h_1,h_2,v^{(2)})=
I^{(2)}_{+,-}(\mathscr K_1 h_1,\mathscr K_1h_2,\mathscr R_{1,1}^{(2)}v^{(2)})+
I^{(2)}_{-,-}(\mathscr K_2h_1,\mathscr K_1h_2,\mathscr R_{2,1}^{(2)}v^{(2)})\\
&\qquad+I^{(2)}_{+,+}(\mathscr K_1h_1,\mathscr K_2h_2,\mathscr R^{(2)}_{1,2}v^{(2)})+
I^{(2)}_{-,+}(\mathscr K_2h_1,\mathscr K_2h_2,\mathscr R^{(2)}_{2,2}v^{(2)})\\
&\quad= I^{(2)}_{-,+}(\mathscr K_1h_2,\mathscr K_1h_1,\mathscr R^{(2)}_{1,1}\Psi_1'v^{(2)})-\eta\int_X (K_1h_1)(x)(K_1h_2)(x)v^{(2)}(x^1,x^1)\, m (dx)\\
&\qquad+I^{(2)}_{-,-}(\mathscr K_1h_2,\mathscr K_2h_1,\mathscr R_{1,2}^{(2)}\Psi_1'v^{(2)})+I^{(2)}_{+,+}(\mathscr K_2h_2,\mathscr K_1h_1,\mathscr R^{(2)}_{2,1}\Psi_1'v^{(2)})\\
&\qquad+I^{(2)}_{+,-}(\mathscr K_2h_2,\mathscr K_2h_1,\mathscr R^{(2)}_{2,2}\Psi_1'v^{(2)})+\int_X (K_2h_1)(x)(K_2h_2)(x)v^{(2)}(x^1,x^1)\, m (dx)\\
&\quad =\mathfrak I^{(2)}_{+,-}(h_2,h_1,\Psi_1'v^{(2)})+
\int_X \big((-\eta K_1^*K_1+K_2^*K_2)h_1\big)(x)h_2(x)v^{(2)}(x^1,x^1)\, m (dx)\\
&\quad =\mathfrak I^{(2)}_{+,-}(h_2,h_1,\Psi_1'v^{(2)})+
\int_X  h_1(x)h_2(x)v^{(2)}(x^1,x^1)\, m (dx).
\end{align*}
\end{proof}

\begin{corollary}\label{uf7ydrt}
Let $k\in\mathbb N$, $k\ge 2$. Let $q\in\mathbb C$ be such that $q\ne1$ and $q^k=1$. Then,   for each $h\in\mathscr H_{\mathbb C}$,
$$\bigg(\int_X  m (dx)\, h(x)D_x^+\bigg)^k=0.$$
\end{corollary}

\begin{proof}
Immediate by Theorems \ref{ftuf} and \ref{tre645}.\end{proof}

\begin{remark}In fact, a more general statement can be shown on the space $\mathscr F_{\mathrm{fin}}^{JQ}(L^2(Z, m ))$: {\it Let the conditions of Corollary~\ref{uf7ydrt} be satisfied. Then,   for any $h_1,h_2\in\mathscr H_{\mathbb C}$, we have}
\begin{equation}\label{tyr75i45}
\bigg(\int_X m (dx)h_1(x)\partial_{x,1}^-+\int_X m (dx)h_2(x)\partial_{x,2}^+\bigg)^k=0.\end{equation}
We leave the (nontrivial) proof of formula \eqref{tyr75i45} to the interested reader.
\end{remark}

\subsection{The associated state}

Let $\mathbf A$ be the complex algebra from Theorem \ref{tre645}.
We define a state $\tau$ on $\mathbf A$ by
\begin{equation}\label{cfyde64e64e3}
\tau(a):=(a\Omega,\Omega)_{\mathscr F^{JQ}(L^2(Z, m ))},\quad a\in\mathbf A.\end{equation}

\begin{theorem}\label{yt8tt}
The state $\tau$ defined by \eqref{cfyde64e64e3} is a gauge-invariant quasi-free state on the $Q$-CR algebra $\mathbf A$. The state $\tau$ is completely determined by the self-adjoint positive operator $T:=K_1^*K_1$ in $\mathscr H_{\mathbb C}$ and satisfies, for any $g,h\in\mathscr H_{\mathbb C}$ and $v^{(2)}\in L^\infty(\R^2,\nu^{(2)}_{+,-})$,
\begin{equation}\label{tyr75i}
\mathbf S^{(1,1)}(g,h,v^{(2)})=\int_X g(x)(Th)(x)v^{(2)}(x^1,x^1)\, m (dx),
\end{equation}
or equivalently
\begin{equation}
\label{re6}
\mathbf S^{(1,1)}[g,h](s_1,s_2)=\chi_{D}(s_1,s_2)\int_{\R^{d-1}}g(s_1,x^2,\dots,x^d)(Th)(s_1,x^2,\dots,x^d)\,dx^2\dotsm dx^d.
\end{equation}
In formula \eqref{re6}, $D=\{(s_1,s_2)\in\R^2\mid s_1=s_2\}$.
\end{theorem}

\begin{remark}Note that, by \eqref{iyt8o}, in the case $\eta<0$, the operator $T$ additionally satisfies $T\le-1/\eta$.
\end{remark}

\begin{remark}
Note that the representation of the $Q$-CR algebra from Theorem~\ref{tre645} depends on operators $K_1$ and $K_2$ satisfying \eqref{iyt8o}, while the state $\tau$ from Theorem~\ref{yt8tt} depends only on $|K_1|:=\sqrt{K_1^*K_1}$.
\end{remark}

\begin{proof}[Proof of Theorem \ref{yt8tt}] For any $g_1,\dots,g_k, h_1,\dots,h_{n}\in\mathscr H_{\mathbb C}$ and   $v^{(k+n)}\in L^\infty(\R^{k+n},\nu^{(k+n)}_{\sharp_1,\dots,\sharp_{k+n}})$ with $\sharp_1=\dots=\sharp_k=+$, $\sharp_{k+1}=\dots=\sharp_{k+n}=-$, we have by \eqref{uyr75yfttr}, \eqref{yufu8rt86rd}, \eqref{iyut8}, \eqref{buyt8t}, and \eqref{cfyde64e64e3},
\begin{align}
&\mathbf S^{(k,n)}(g_k,\dots,g_1,h_1,\dots,h_n,v^{(k+n)})\notag\\
& \quad =
\left(\mathfrak I^{(k+n)}_{\sharp_1,\dots,\sharp_{k+n}}\big(g_k,\dots,g_1,h_1,\dots,h_n,v^{(k+n)}\big)\Omega,\Omega\right)_{\mathscr F^{JQ}(L^2(Z, m ))}\notag\\
& \quad=
\left(I^{(k+n)}_{\sharp_{k+1},\dots,\sharp_{k+n},\sharp_1,\dots,\sharp_k}\big(\mathscr K_1g_k,\dots,\mathscr K_1g_1,\right.\notag\\
&\qquad\qquad\left.\vphantom{I^{(k+n)}_{\sharp_{k+1},\dots,\sharp_{k+n},\sharp_1,\dots,\sharp_k}}
\mathscr K_1h_1,\dots,\mathscr K_1h_n,\mathscr R^{(k+n)}_{1,\dots,1}v^{(k+n)}\big)\Omega,\Omega\right)_{\mathscr F^{JQ}(L^2(Z,{k}))}.\label{gftft}\end{align}
Hence $\mathbf S^{(k,n)}\equiv 0$ if $k\ne n$, and for $k=n$ we get from \eqref{gftft}:
\begin{align}
&\mathbf S^{(n,n)}(g_n,\dots,g_1,h_1,\dots,h_n,v^{(2n)})\notag\\
&\quad =
\left(I^{(2n)}_{\sharp_{n+1},\dots,\sharp_{2n},\sharp_1,\dots,\sharp_n}\big( K_1g_n,\dots,K_1g_1, K_1h_1,\dots,K_1h_n,v^{(2n)}\big)\Omega,\Omega\right)_{\mathscr F^{Q}(L^2(X, m ))}.\label{tyde6eftt}
\end{align}
Note that formulas \eqref{tyr75i}, \eqref{re6} trivially follow from \eqref{tyde6eftt} with $n=1$.

For the general $n\in\mathbb N$, analogously to the proof of Lemma~\ref{tuyfr7}, it suffices to  consider the case where
\begin{equation}\label{tyyr7r7o8}
 v^{(2n)}(s_1,\dots,s_{2n})=u_n(s_1)\dotsm u_1(s_{n}) w_1(s_{n+1})\dotsm w_n(s_{2n}),\end{equation}
where $u_1,\dots,u_n,w_1,\dots,w_n$ are indicator functions of sets from $\mathscr B(\R)$. Denote
$$g_i'(x):=g_i(x)u_i(x^1),\quad h_i'(x):=h_i(x) w_i(x^1),\quad x\in X,\ i=1,\dots,n.$$
We get from (A1), Remarks~\ref{yf7}, \ref{yut86}, and formulas \eqref{hgyi8},  \eqref{tyde6eftt}, \eqref{tyyr7r7o8}
\begin{align}
&\mathbf S^{(n,n)}(g_n,\dots,g_1,h_1,\dots,h_n,v^{(2n)})\notag\\
&\quad =
\left(I^{(2n)}_{\sharp_{n+1},\dots,\sharp_{2n},\sharp_1,\dots,\sharp_n}\big( K_1g'_n,\dots,K_1g'_1, K_1h'_1,\dots,K_1h'_n,1\big)\Omega,\Omega\right)_{\mathscr F^{Q}(L^2(X, m ))}\notag\\
&\quad=\bigg(\int_X  m (dx_n)\, (K_1g_n')(x_n)\partial_{x_n}^-\notag\\
&\qquad \dotsm
\int_X  m (dx_1)\, (K_1g_1')(x_1)\partial_{x_1}^-\,\big( (K_1 h'_1)\circledast\dotsm \circledast (K_1 h'_n)\big)
 \Omega,\Omega\bigg)_{\mathscr F^{Q}(L^2(X, m ))}\notag\\
 &\quad=\big((K_1 h'_1)\circledast\dotsm \circledast (K_1 h'_n),
 (\overline{K_1 g'_1})\circledast\dotsm \circledast (\overline{K_1 g'_n})
 \big)_{\mathscr F^{Q}(L^2(X, m ))}\notag\\
 &\quad=n!\left(P_n (K_1 h'_1)\otimes\dotsm \otimes (K_1 h'_n),(K_1 \overline{g'_1})\otimes\dotsm \otimes (K_1 \overline{g'_n})\right)_{\mathscr H_{\mathbb C}^{\otimes n}}\notag\\
 &\quad=\sum_{\pi\in S_n}\left(
 \big[
 (K_1h'_{\pi(1)})\otimes\dots\otimes (K_1h'_{\pi(n)})
 \big]Q_\pi,(K_1 \overline{g'_1})\otimes\dotsm \otimes (K_1 \overline{g'_n})
 \right)_{\mathscr H_{\mathbb C}^{\otimes n}}\notag\\
 &\quad=\sum_{\pi\in S_n}\left(
 (K_1^*\otimes\dots\otimes K^*_1)Q_\pi(K_1\otimes\dots\otimes K_1)
 (h'_{\pi(1)}\otimes\dots\otimes h'_{\pi(n)}),
 \overline{g'_1}\otimes\dots\otimes \overline{g'_n}
 \right)_{\mathscr H_{\mathbb C}^{\otimes n}}\notag\\
 &\quad=\sum_{\pi\in S_n}\left( Q_\pi \big[(Th'_{\pi(1)})\otimes\dots\otimes   (Th'_{\pi(n)})\big],\overline{g'_1}\otimes\dots\otimes \overline{g'_n} \right)_{\mathscr H_{\mathbb C}^{\otimes n}}\notag\\
 &\quad=\sum_{\pi\in S_n}\int_{X^n}
 Q_\pi(x_1^1,\dots,x_n^1) (Th_{\pi(1)})(x_1)g_1(x_1)\dotsm  (Th_{\pi(n)})(x_n)g_n(x_n)\notag\\
 &\qquad\times u_1(x_1^1)\dotsm u_n(x_n^1)w_{\pi(1)}(x_1^1)\dotsm w_{\pi(n)}(x_n^1)\, m ^{\otimes n}(dx_1\dotsm dx_n)\notag\\
 &\quad =\sum_{\pi\in S_n}\int_{X^n}\bigg(\prod_{i=1}^n g_i(x_i)(Th_{\pi(i)})(x_i)\bigg)Q_\pi(x_1^1,\dots,x_n^1)\notag\\
&\qquad  \times v^{(2n)}(x_{n}^1,\dots,x_1^1,x_{\pi^{-1}(1)}^1,\dots,x_{\pi^{-1}(n)}^1)\, m ^{\otimes n}(dx_1\dotsm dx_n)\notag\\
&\quad= \sum_{\pi\in S_n}\int_{\R^{n}}\left(\prod_{i=1}^n\int_{\R^{d-1}} g_i(s_i,x^2,\dots,x^d)(Th_{\pi(i)})(s_i,x^2,\dots,x^d)\,dx^2\dotsm dx^d\right)
\notag\\
&\qquad\times
Q_\pi(s_1,\dots,s_{n})v^{(2n)}(s_n,\dots,s_1,s_{\pi^{-1}(1)},\dots,s_{\pi^{-1}(n)})
\,ds_1\dotsm ds_n\notag\\
&\quad= \sum_{\pi\in S_n}\int_{\R^{2n}}
\left(\prod_{i=1}^n \chi_D(s_i,s_{n+\pi(i)})
\int_{\R^{d-1}} g_i(s_i,x^2,\dots,x^d)(Th_{\pi(i)})(s_i,x^2,\dots,x^d)\,dx^2\dotsm dx^d\right)
\notag\\
&\qquad\times
Q_\pi(s_1,\dots,s_{n})v^{(2n)}(s_n,\dots,s_1,s_{n+1},\dots,s_{2n})
\,\nu^{(2n)}_{\sharp_1,\dots,\sharp_{2n}}(ds_1\dotsm ds_{2n})\notag\\
 &\quad= \int_{\R^{2n}}\sum_{\pi\in S_n}
\left(\prod_{i=1}^n \mathbf S^{(1,1)}[g_i,h_{\pi(i)}](s_i,s_{n+\pi(i)}) \right)
\notag\\
&\qquad\times
Q_\pi(s_1,\dots,s_{n})v^{(2n)}(s_n,\dots,s_1,s_{n+1},\dots,s_{2n})
\,\nu^{(2n)}_{\sharp_1,\dots,\sharp_{2n}}(ds_1\dotsm ds_{2n}).\label{dre64}
  \end{align}
Hence, by \eqref{utfr75re}, the theorem follows.
\end{proof}

The following corollary immediately follows from \eqref{re6} and \eqref{dre64}. The reader is advised to compare it with formulas \eqref{cfyd6}--\eqref{yiut}.

\begin{corollary}
Let $v^{(2n)}:\R^{2n}\to\mathbb C$ be identically equal to 1.
Then, for any $g_1,\dots,g_n,\linebreak
h_1,\dots,h_{n}\in\mathscr H_{\mathbb C}$,
\begin{align*}
&\mathbf S^{(n,n)}(g_n,\dots,g_1,h_1,\dots,h_n,1)\\
&\quad =\sum_{\pi\in S_n}
\int_{X^n}\left(\prod_{i=1}^n g_i(x_i)(Th_{\pi(i)})(x_i)\right)Q_\pi(x_1,\dots,x_n)\, m ^{\otimes n}(dx_1\dotsm dx_n).
\end{align*}
\end{corollary}

\begin{corollary}
Let $\eta\ge0$. Let $T$ be a continuous linear operator in $\mathscr H$ that is self-adjoint and positive.  Assume that, for any bounded measurable function $\psi:\R\to\R$, the operator $T$ commutes with $M_\psi$ (see condition (C)). Extend $T$ by linearity to $\mathscr H_{\mathbb C}$.
Then there exists a gauge-invariant quasi-free state $\tau$ on the $Q$-CR algebra $\mathbf A$
that satisfies \eqref{tyr75i}.

If $\eta<0$, the latter statement remains true if the operator $T$ additionally satisfies $0\le T\le -1/\eta$.
\end{corollary}

\begin{proof} Choose $K_1:=\sqrt{T}$, $K_2:=\sqrt{\mathbf 1+\eta T}$. Note that $K_1$ and $K_2$ satisfy condition~(C), see Remark~\ref{uyr75khuyh}. Now the corollary follows from Theorem~\ref{yt8tt}.
\end{proof}

\section{Particle density}\label{crds54w}

Let operators $\partial_x^+$, $\partial_x^-$ ($x\in X$) satisfy the ACR. We heuristically define the particle density by
$$\rho(x):=\partial_x^+\partial_x^-, \quad x\in X.$$
It follows from the $Q$-CR that these operators commute, cf.\ \cite{GM,Goldin_Sharp}. Indeed, for any $x,y\in X$,
\begin{align*}
\rho(x)\rho(y)&=\partial_x^+\partial_x^-\partial_y^+\partial_y^-=\delta(x-y)\partial_x^+\partial_y^-+Q(x,y)\partial_x^+\partial_y^+\partial_x^-\partial_y^-\\
&=\delta(x-y)\partial_x^+\partial_x^-+Q(y,x)\partial_y^+\partial_x^+\partial_y^-\partial_x^-\\
&=\delta(x-y)\partial_x^+\partial_x^-+\partial_y^+\partial_y^-\partial_x^+\partial_x^--\delta(x-y)\partial_x^+\partial_x^-=\rho(y)\rho(x).
\end{align*}

In this section, we will study the particle density corresponding to the gauge-invariant quasi-free state from Theorem~\ref{yt8tt} with $T=\kappa^2\mathbf 1$, where $\kappa>0$ is a constant. In the case $\eta<0$, we additionally assume that $\kappa^2<1/\eta$. Thus, we set $K_1=\kappa\mathbf 1$ and $K_2=\sqrt{1+\eta\kappa^2}\,\mathbf 1$. We will see below that, in order to properly define $\int_X m (dx)\,\varphi(x)\rho(x)$ for a test function $\varphi:X\to\R$, we will need a certain renormalization.

We will also assume that $\eta=\Re q$. Note that, with this choice of $\eta$, we get
$$Q(x,x)=\frac12\bigg(\lim_{x\to y,\ x^1>y^1}Q(x,y)+\lim_{x\to y,\ x^1<y^1}Q(x,y\bigg),\quad x\in X.$$
We will use this value of $Q(x,x)$ as a `limiting value' of $Q(x,y)$ ($x_1\ne y_1$) when performing renormalization.

\subsection{Renormalization}

We start with heuristic calculations. By \eqref{vgydy6d6} and \eqref{guf7} with the above choice of the operators $K_1$ and $K_2$, we get
\begin{align*}
D_x^+&=\kappa\partial_{x,1}^-+\sqrt{1+\eta\kappa^2}\,\partial_{x,2}^+,\\
D_x^-&=\kappa\partial_{x,1}^++\sqrt{1+\eta\kappa^2}\,\partial_{x,2}^-.
\end{align*}
Hence, the corresponding particle density is given by
\begin{equation}\label{vcyre6jr876r7}
\rho(x)=D_x^+D_x^-=\kappa\sqrt{1+\eta\kappa^2}\, R(x),\end{equation}
where
\begin{equation}\label{fufuyuyf}
 R(x):=\partial_{x,2}^+\partial_{x,1}^++\partial_{x,1}^-\partial_{x,2}^-+\beta^{-1}\partial_{x,1}^-\partial_{x,1}^++\beta\partial_{x,2}^+\partial_{x,2}^-\end{equation}
with $\beta:=\sqrt{\eta+\kappa^{-2}}$.
We denote by $\mathscr D$ the space of all real-valued infinitely differentiable functions on $X$ with compact support. For each $\varphi\in\mathscr D$, we denote $\rho(\varphi):=\int_X m (dx)\,\varphi(x)\rho(x)$. We denote by $\mathfrak R$ the real commutative algebra generated by $\rho(\varphi)$ ($\varphi\in\mathscr D$) and constants. The involution on this algebra is the indentity mapping.
We would like to define a vacuum state $\tau$ on $\mathfrak R$ analogously to \eqref{cfyde64e64e3}. However, we are not able to intepret $\rho(\varphi)$ as a linear operator in $\mathscr F^{JQ}(L^2(Z, m ))$. So we need a proper renormalization.
For this, as discussed in Introduction, we will use Ivanov's formula \cite{Ivanov} in the form $\delta^2=\delta.$

Below we will denote by $\odot$ symmetric tensor product. Let $\mathscr D^{\odot_{\mathrm{a}}n}$ denote the $n$th symmetric algebraic tensor power of $\mathscr D$, i.e., the vector space of finite sums of functions on $X^n$ of the form $f_1\odot\dots\odot f_n$, where $f_1,\dots,f_n\in\mathscr D$. For each $f^{(n)}\in \mathscr D^{\odot_{\mathrm{a}}n}$, we denote
$$ W(f^{(n)}):=\int_{X^n} m ^{\otimes n}(dx_1\dotsm dx_n)f^{(n)}(x_1,\dots,x_n)\partial_{x_1,2}^+\partial_{x_1,1}^+\dotsm \partial_{x_n,2}^+\partial_{x_n,1}^+\Omega.$$
We stress that $W(f^{(n)})$ is treated as a formal expression.
Note that, for $f_1,\dots,f_n\in\mathscr D$,
$$ W(f_1\odot\dots\odot f_n)=\bigg(\int_X m (dx_1)f_1(x_1)\partial_{x_1,2}^+\partial_{x_1,1}^+\dotsm \int_X m (dx_n)f_n(x_n)\partial_{x_n,2}^+\partial_{x_n,1}^+\bigg)\Omega,$$
where we used that $\partial_{x_i,2}^+\partial_{x_i,1}^+$ and $\partial_{x_j,2}^+\partial_{x_j,1}^+$ commute for $i\ne j$.
We also set $\mathscr D^{\odot_{\mathrm{a}}0}:=\R$ and for $f^{(0)}\in \R$, we set $W(f^{(0)}):=f^{(0)}\Omega$.
We denote by $\mathscr F_{\mathrm{fin}}(\mathscr D)$ the real linear space of vectors of the form
$$ F=(f^{(0)},f^{(1)},\dots,f^{(n)},0,0,\dots),$$
where $f_i\in \mathscr D^{\odot_{\mathrm{a}}i}$. For such $F\in \mathscr F_{\mathrm{fin}}(\mathscr D)$, we denote
$$W(F):=\sum_{i=0}^n W(f^{(i)}).$$

The following proposition will be central for our considerations.

\begin{proposition}\label{yur75ri7}
For each $\varphi\in\mathscr D$, we define a linear operator
$$\hat R(\varphi):\mathscr F_{\mathrm{fin}}(\mathscr D)\to\mathscr F_{\mathrm{fin}}(\mathscr D)$$ by
\begin{equation}\label{ftd5ed5d}\hat R(\varphi):=a^+(\varphi)+(\beta+\beta^{-1}\eta)a^0(\varphi)+a_1^-(\varphi)+\eta a_2^-(\varphi)+\beta^{-1}\int_X\varphi(x)\, m (dx).\end{equation}
Here $a^+(\varphi)$ is the symmetric creation operator: for $f^{(n)}\in \mathscr D^{\odot_{\mathrm{a}}n}$ we have $a^+(\varphi)f^{(n)}:=\varphi\odot f^{(n)}$; $a^0(\varphi)$ is the neutral operator:
for $f_1,\dots,f_n\in\mathscr D$
$$a^0(\varphi)f_1\odot\dots\odot f_n:=\sum_{i=1}^n f_1\odot\dots\odot (\varphi f_i)\odot\dots\odot f_n;$$
$a_1^-(\varphi)$ is the annihilation operator:
$$a_1^-(\varphi)f_1\odot\dots\odot f_n:=\sum_{i=1}^n \int_X \varphi(x)f_i(x)\, m (dx)\, f_1\odot\dots\odot \check f_i\odot\dots\odot f_n,$$
where $\check f_i$ denotes the absence of $f_i$; and $a_2^-(\varphi)$ is an annihilation operator which acts as follows:
$$a_2^-(\varphi)f_1\odot\dots\odot f_n:=\sum_{i=1}^n\sum_{j\ne i}
f_1\odot\dots\odot f_{\min\{i,j\}-1}\odot (\varphi f_if_j)\odot\dots
\odot\check f_{\max\{i,j\}}\odot\dots\odot f_n.$$
Let
$$R(\varphi):=\int_X  m (dx)\varphi(x)R(x)=\left(\kappa\sqrt{1+\eta\kappa^2}\right)^{-1}\rho(\varphi).$$  Then, the $JQ$-commutation relations satisfied by $\partial_{x,i}^+$, $\partial_{x,i}^-$, the condition $\partial_{x,i}^-\Omega=0$, and the renormalization formula $\delta^2=\delta$ imply that, for each $F\in\mathscr F_{\mathrm{fin}}(\mathscr D)$,
\begin{equation}\label{tye65re}R(\varphi)W(F)=W(\hat R(\varphi)F).
\end{equation}
\end{proposition}

\begin{proof}
We trivially get
\begin{equation}\label{ytd6e6iff}
 \int_X m (dy)\varphi(y) \partial_{y,2}^+\partial_{y,1}^+\, W(F)=W(a^+(\varphi) F).\end{equation}
For each $f^{(n)}\in \mathscr D^{\odot_{\mathrm{a}}n}$, we have
\begin{align}
&\int_X m (dy)\varphi(y) \partial_{y,1}^-\partial_{y,2}^-\, W(f^{(n)})
 =
\int_{X^{n+1}} m ^{\otimes(n+1)}(dy\,dx_1\dotsm dx_n)\varphi(y)f^{(n)}(x_1,\dots,x_n) \notag\\
&\qquad\times \partial_{y,1}^-\partial_{y,2}^- \partial_{x_1,2}^+\partial_{x_1,1}^+\dotsm \partial_{x_n,2}^+\partial_{x_n,1}^+\Omega.\label{tsers}\end{align}
Note that
\begin{align*}
\partial_{y,2}^-\partial_{x_i,2}^+\partial_{x_i,1}^+&=\delta(y-x_i)\delta_{x_i,1}^++Q(y,x_i)\partial_{x_i,2}^+\partial_{y,2}^-\partial_{x_i,1}^+\\
&=\delta(y-x_i)\delta_{x_i,1}^++\partial_{x_i,2}^+\partial_{x_i,1}^+\partial_{y,2}^-,
\end{align*}
and for $i\ne j$
\begin{equation}\label{ufr76}
\partial_{x_i,1}^+\partial_{x_j,2}^+\partial_{x_j,1}^+=\partial_{x_j,2}^+\partial_{x_j,1}^+\partial_{x_i,1}^+.\end{equation}
Hence,
\begin{align}
&\partial_{y,1}^-\partial_{y,2}^- \partial_{x_1,2}^+\partial_{x_1,1}^+\dotsm \partial_{x_n,2}^+\partial_{x_n,1}^+\Omega\notag\\
&\quad =\sum_{i=1}^n\delta(y-x_i)\partial_{y,1}^- \partial_{x_1,2}^+\partial_{x_1,1}^+\dotsm \big(\partial_{x_i,2}^+,\partial_{x_i,1}^+\big)\check{}\dotsm
\partial_{x_n,2}^+\partial_{x_n,1}^+\partial_{x_i,1}^+\Omega.\label{tydy6}
\end{align}
Here and below $(\dotsm)\check{}$ denotes the absence of the corresponding term. Similarly, we have
\begin{align}
\partial_{y,1}^-\partial_{x_j,2}^+\partial_{x_j,1}^+&=Q(x_j,y)\partial_{x_j,2}^+\partial_{y,1}^-\partial_{x_j,1}^+\notag\\
&=Q(x_j,y)\partial_{x_j,2}^+\delta(y-x_j)+\partial_{x_j,2}^+\partial_{x_j,1}^+\partial_{y,1}^-\notag\\
&=\eta\partial_{x_j,2}^+\delta(y-x_j)+\partial_{x_j,2}^+\partial_{x_j,1}^+\partial_{y,1}^-,\label{yur76}
\end{align}
Thus, using \eqref{ufr76} and \eqref{yur76}, we  continue \eqref{tydy6} as follows:
\begin{align}
&=\eta\sum_{i=1}^n\sum_{j\ne i} \delta(y-x_i)\delta(y-x_j)\partial_{x_1,2}^+\partial_{x_1,1}^+\dotsm \big(\partial_{x_{\max\{i,j\}},2}^+\partial_{x_{\max\{i,j\}},1}^+\big)\check{}\dotsm \partial_{x_n,2}^+\partial_{x_n,1}^+\Omega\notag\\
&\quad+\sum_{i=1}^n\delta(y-x_i)\partial_{x_1,2}^+\partial_{x_1,1}^+\dotsm \big(\partial_{x_i,2}^+,\partial_{x_i,1}^+\big)\check{}\dotsm
\partial_{x_n,2}^+\partial_{x_n,1}^+\partial_{y,1}^-\partial_{x_i,1}^+\Omega.\label{fyd6}
\end{align}
We also note that
\begin{equation}\label{uyfd6e}
\delta(y-x_i)\partial_{y,1}^-\partial_{x_i,1}^+\Omega=\delta^2(y-x_i)\Omega=\delta(y-x_i)\Omega.
\end{equation}
Hence, by \eqref{tydy6}, \eqref{fyd6}, and \eqref{uyfd6e},
\begin{equation}\label{ftutr}
\int_X m (dy)\varphi(y) \partial_{y,1}^-\partial_{y,2}^-\, W(f^{(n)})=
W\left(\big(a_1^-(\varphi)+\eta a_2^-(\varphi)\big)f^{(n)}\right).\end{equation}

Similarly,
\begin{align*}
&\partial_{y,2}^+\partial_{y,2}^-\partial_{x_1,2}^+\partial_{x_1,1}^+\dotsm \partial_{x_n,2}^+\partial_{x_n,1}^+\Omega\\
&\quad= \partial_{y,2}^+\sum_{i=1}^n\delta(y-x_i)\partial_{x_1,2}^+\partial_{x_1,1}^+\dotsm \big(\partial_{x_i,2}^+\big)\check{}\,\partial_{x_i,1}^+\dotsm \partial_{x_n,2}^+\partial_{x_n,1}^+\Omega\\
&\quad= \sum_{i=1}^n\delta(y-x_i)\partial_{x_1,2}^+\partial_{x_1,1}^+\dotsm \big(\partial_{x_i,2}^+\big)\check{}\,\partial_{y,2}^+\partial_{x_i,1}^+\dotsm \partial_{x_n,2}^+\partial_{x_n,1}^+\Omega\\
&\quad= \sum_{i=1}^n\delta(y-x_i)\partial_{x_1,2}^+\partial_{x_1,1}^+\dotsm \partial_{x_i,2}^+\partial_{x_i,1}^+\dotsm \partial_{x_n,2}^+\partial_{x_n,1}^+\Omega,
\end{align*}
which implies
\begin{equation}\label{fdq5ru6tr765}
\int_X m (dy)\varphi(y)\partial_{y,2}^+\partial_{y,2}^- W(f^{(n)})=W(a^0(\varphi)f^{(n)}).\end{equation}

Finally,
\begin{align*}
&\partial_{y,1}^-\partial_{y,1}^+\partial_{x_1,2}^+\partial_{x_1,1}^+\dotsm \partial_{x_n,2}^+\partial_{x_n,1}^+\Omega\\
&\quad =\partial_{y,1}^-\partial_{x_1,2}^+\partial_{x_1,1}^+\dotsm \partial_{x_n,2}^+\partial_{x_n,1}^+\partial_{y,1}^+\Omega\\
&\quad=\eta\sum_{i=1}^n\delta(y-x_i) \partial_{x_1,2}^+\partial_{x_1,1}^+\dotsm \partial_{x_i,2}^+ \big(\partial_{x_i,1}^+\big)\check{}\,\dotsm \partial_{x_n,2}^+\partial_{x_n,1}^+\partial_{y,1}^+\Omega\\
&\qquad+\partial_{x_1,2}^+\partial_{x_1,1}^+\dotsm \partial_{x_n,2}^+\partial_{x_n,1}^+\partial^-_{y,1}\partial_{y,1}^+\Omega\\
&\quad= \eta\sum_{i=1}^n\delta(y-x_i)\partial_{x_1,2}^+\partial_{x_1,1}^+\dotsm \partial_{x_i,2}^+\partial_{x_i,1}^+\dotsm \partial_{x_n,2}^+\partial_{x_n,1}^+\Omega\\
&\qquad+\partial_{x_1,2}^+\partial_{x_1,1}^+\dotsm \partial_{x_n,2}^+\partial_{x_n,1}^+\Omega.
\end{align*}
Here we used that
\begin{align*}
\int_X m (dy)\varphi(y)\partial^-_{y,1}\partial_{y,1}^+\Omega&=\int_{X^2} m ^{\otimes 2}(dy\,du)\delta(y-u)\varphi(y)\partial^-_{y,1}\partial_{u,1}^+\Omega\\
&=\int_{X^2} m ^{\otimes 2}(dy\,du)\delta^2(y-u)\varphi(y)\Omega\\
&=\int_{X^2} m ^{\otimes 2}(dy\,du)\delta(y-u)\varphi(y)\Omega=\int_X\varphi(y)\, m (dy)\Omega.
\end{align*}
Thus,
\begin{equation}\label{yu567}
\int_X m (dy)\,\varphi(y)\partial^-_{y,1}\partial_{y,1}^+ W(f^{(n)})=
W(\eta a^0(\varphi)f^{(n)})+\int_X\varphi(y)\, m (dy) W(f^{(n)}).\end{equation}
Now, formula \eqref{tye65re} follows from \eqref{fufuyuyf}, \eqref{ytd6e6iff}, \eqref{ftutr}--\eqref{yu567}.
\end{proof}

\begin{corollary}\label{yuro68}
We have
$$\{r\Omega\mid r\in\mathfrak R\}=\{W(F)\mid F\in\mathscr F_{\mathrm{fin}}(\mathscr D)\}.$$
More precisely, for each $F\in\mathscr F_{\mathrm{fin}}(\mathscr D)$, there exists a unique $r\in\mathfrak R$ such that $W(F)=r\Omega$. This correspondence is given through formula \eqref{tye65re}. Vice versa, for each $r\in\mathfrak R$, there exists a unique $F\in\mathscr F_{\mathrm{fin}}(\mathscr D)$ such that $W(F)=r\Omega$ holds.
\end{corollary}

\begin{proof} By Proposition~\ref{yur75ri7}, we have, for $\varphi\in\mathscr D$,
$$R(\varphi)\Omega=W(\varphi)+\beta^{-1}\Omega.$$
Hence
\begin{equation}\label{tyde75i}
W(\varphi)=(R(\varphi)-\beta^{-1})\Omega.\end{equation}
By \eqref{ftd5ed5d} and \eqref{tye65re}, for any $n\ge2$ and $f_1,\dots,f_n\in\mathscr D$,
\begin{align}
&W(f_1\odot\dots\odot f_n)=R(f_1)W(f_2\odot\dots\odot f_n)\notag\\
&\quad-W\left(\bigg[(\beta+\beta^{-1}\eta)a^0(f_1)+a_1^-(f_1)+\eta a_2^-(f_1)+\beta^{-1}\int_X f_1(x)\, m (dx)\bigg]f_2\odot \dots\odot f_n\right).\label{tfur7}
\end{align}
Hence, for $F\in\mathscr F_{\mathrm{fin}}(\mathscr D)$, a unique representation of $W(F)$ as $r\Omega$ ($r\in\mathfrak R$) follows by induction on $n$ and formulas \eqref{tyde75i}, \eqref{tfur7}. The converse statement  follows immediately from Proposition~\ref{yur75ri7} by induction.
\end{proof}

Corollary \ref{yuro68}  implies that, for each $r\in\mathfrak R$, there exists a unique vector
\begin{equation}\label{drte6u}
F=(f^{(0)},f^{(1)},\dots,f^{(n)},0,0,\dots)\in\mathscr F_{\mathrm{fin}}(\mathscr D)\end{equation}such that
$$
r\Omega=f^{(0)}\Omega+\sum_{i=1}^n\int_{X^i} m ^{\otimes i}(dx_1\dotsm dx_i)f^{(i)}(x_1,\dots,x_i) \partial_{x_1,2}^+\partial_{x_1,1}^+\dotsm \partial_{x_i,2}^+\partial_{x_i,1}^+\Omega.$$
Thus, we can rigorously define a linear mapping $\tau:\mathfrak R\to\R$ by $\tau(r):=f^{(0)}$. However, since we used the renormalization, from our construction of $\tau$ is  not {\it a priori\/} clear whether $\tau$ is positive definite, i.e., whether $\tau(r^2)\ge0$ for each $r\in\mathfrak R$.  So, our next aim is to decide whether positive definiteness holds.

\subsection{Measure-valued L\'evy processes and positive definiteness of $\tau$}

We start with preliminaries on measure-valued L\'evy process. For more detail, see e.g.\ \cite{Kal,Kingman,Kingman1,KLV,HKPR}.

 Let $\mathbb M(X)$ denote the space of all Radon measures on $X$. We equip $\mathbb M(X)$ with the topology of vague convergence. Let $\mathscr B(\mathbb M(X))$ denote the corresponding Borel $\sigma$-algebra on $\mathbb M(X)$. Let $(\Omega,\mathscr A,P)$ be a provability space. A {\it random measure on $X$} is a measurable mapping $\gamma:\Omega\to\mathbb M(X)$. A {\it competely random measure\/} is a random measure $\gamma$ such that, for any  mutually disjoint  sets $A_1,\dots,A_n\in\mathscr B_0(X)$, the random variables $\gamma(A_1),\dots,\gamma(A_n)$ are independent.
 Here $\mathscr B_0(X)$ denotes the subset of $\mathscr B(X)$ that consists of all bounded sets from $\mathscr B(X)$.
  A {\it measure-valued L\'evy process\/} is a completely random measure $\gamma$ such that, for any sets $A_1,A_2\in\mathscr B_0(X)$ with $ m (A_1)= m (A_2)$, the random variables $\gamma(A_1)$ and $\gamma(A_2)$ are identically distributed.

 It follows from \cite{Kingman} that a random measure $\gamma$ is a measure-valued L\'evy process if and only if there exist a constant $c\ge0$ and a measure $\zeta$ on $\R_+:=(0,\infty)$ satisfying
 $$ \int_{\R_+}\min\{s,1\}\,\zeta(ds)<\infty$$
 such that the Laplace transform of $\gamma$ is given by
 \begin{equation}\label{bhd67}
 \mathbb E\left(e^{\langle f,\gamma\rangle}\right)=\exp\bigg[
 c\int_X f(x)\, m (dx)+\int_X\int_{\R_+} (e^{sf(x)}-1)\,\zeta(ds)\, m (dx) \bigg],\quad f\in C_0(X),\ f\le0.\end{equation}
Here $C_0(X)$ denotes the space of real-valued continuous functions on $X$ with compact support and $\langle f,\gamma\rangle:=\int_X f(x)\,\gamma(dx)$. The measure $\zeta$ is called the {\it L\'evy measure of the
measure-valued L\'evy process $\gamma$}.

Let $\mathbb K(X)$ denote the subset of $\mathbb M(X)$ consisting of all discrete Radon measures on $X$, i.e., Radon measures of the form
$\sum_{i\in I}s_i\delta_{x_i}$, where the set $I$ is either finite or countable, $s_i>0$,  and $\delta_{x_i}$ denotes here the Dirac measure with mass at $x_i$. We also assume that $x_i\ne x_j$ if $i\ne j$.  A {\it discrete random measure\/} is a random measure which takes values in $\mathbb K(X)$ a.s. Each measure-valued L\'evy process $\gamma$ for which $c=0$ in formula \eqref{bhd67} is a discrete random measure.

The space $\ddot\Gamma(X)$ of multiple configurations in $X$ is the subset of $\mathbb K(X)$ which consists of all Radon measures of the form $\sum_{i\in I}s_i\delta_{x_i}$  with $s_i\in\mathbb N$. A {\it point process
on $X$} is a random measure $\gamma$ which takes values in $\ddot\Gamma(X)$ a.s.

The configuration space $\Gamma(X)$ is defined as the subset of $\ddot\Gamma(X)$ which consists of all Radon measures of the form $\sum_{i\in I}\delta_{x_i}$. Each Radon measure $\sum_{i\in I}\delta_{x_i}$ can be identified with the locally finite set $\{x_i\mid i\in I\}\subset X$. A {\it simple point process on $X$} is a random measure $\gamma$ which takes values in $\Gamma(X)$ a.s.

It should be noted that if $c=0$ and  $\zeta(\R_+)<\infty$, the measure-valued L\'evy process with Fourier transform \eqref{bhd67}
has the property that a.s.\ $\gamma=\sum_{i\in I}s_i\delta_{x_i}$, where the set $\{x_i\mid i\in I\}$ is locally finite, i.e., a configuration in $X$. On the other hand, if $\zeta(\R_+)=\infty$, the set  $\{x_i\mid i\in I\}$ is a.s.\ dense in $X$.

By choosing $c=0$ and the measure $\zeta$ in \eqref{bhd67} to be $z\delta_1$ with $z>0$, one obtains the simple point process $\gamma$ with Laplace transform
$$\mathbb E\left(e^{\langle f,\gamma\rangle}\right)=\exp\bigg[
\int_X (e^{f(x)}-1)\,z m (dx) \bigg].$$
This $\gamma$ is called the {\it Poisson point process with intensity measure $z m $}, since for each $A\in\mathscr B_0(X)$, the random variable $\gamma(A)$ has Poisson distribution with parameter $z m (A)$.

\begin{theorem}\label{cyd5aq42q}
The functional $\tau$ on the real algebra $\mathfrak R$ is positive definite (i.e., $\tau(r^2)\ge0$ for each $r\in\mathfrak R$) if and only if $\eta\ge0$.

Furthermore, if $\eta=0$, then
\begin{equation}\label{cye57}
\tau(\rho(f_1)\dotsm\rho(f_n))=\mathbb E\big(\langle f_1,\gamma\rangle\dotsm \langle f_1,\gamma\rangle\big),\quad f_1,\dots,f_n\in\mathscr D, \end{equation}
where $\gamma$ is the Poisson point process on $X$ with intensity measure $\kappa^2 m $.

If $\eta>0$, then \eqref{cye57} holds with $\gamma$ being a negative binomial point process. More precisely, $\gamma$ is the measure-valued L\'evy process with Laplace  transform \eqref{bhd67} in which $c=0$ and
\begin{equation}\label{crte64}
\zeta=\frac1{\eta}\sum_{k=1}^\infty \bigg(\frac \eta{\eta+\kappa^{-2}}\bigg)^k\frac1k\,\delta_k.\end{equation}
For each $A\in\mathscr B_0(X)$, the distribution of the random variable $\gamma(A)$ is the negative binomial distribution
\begin{equation}\label{yur7e65ew}
 (1+\kappa^2\eta)^{- m (A)/\eta}\sum_{n=0}^\infty \bigg(\frac{\eta}{\eta+\kappa^{-2}}\bigg)^n\frac{\left(\frac{ m (A)}\eta\right)^{(n)}}{n!}\,\delta_n.\end{equation}
Here, we used the standard symbol $a^{(n)}:=a(a+1)(a+2)\dotsm (a+n-1)$, the so-called rising factorial.
\end{theorem}

The proof of Theorem \ref{cyd5aq42q} is based on the property of orthogonal polynomials of a L\'evy white noise proved in \cite[Theorem 2.1 and Corollaries~2.1, 2.3]{Lyt} and \cite{BLM}, see also \cite[Theorem~1.2]{BLR}. We will now briefly explain this result in the special case of a measure-valued L\'evy process.

Let $\gamma$ be a measure-valued L\'evy process such that $c=0$ in \eqref{bhd67} (so that $\gamma$ is a discrete random measure).
We denote $\zeta'(ds):=s^2\zeta(ds)$, the so-called {\it Kolmogorov measure of the measure-valued L\'evy process $\gamma$}. We assume that $\zeta'$ is a probability measure on $\R_+$, and furthermore,
\begin{equation}\label{dre64e}\int_{\R_+}e^{\varepsilon s}\,\zeta'(ds)<\infty\quad\text{for some $\varepsilon>0$}.\end{equation}
 The latter assumption implies that the set of polynomials is dense in $L^2(\R_+,\zeta')$. If the support of the measure $\zeta'$ has infinitely many points, we will denote by $(p_k)_{k=0}^\infty$ the system of monic polynomials that are orthogonal with respect to $\zeta'$. These polynomials satisfy the recurrence relation
\begin{equation}\label{tre6}
sp_k(s)=p_{k+1}(s)+b_kp_k(s)+a_kp_{k-1}(s),\quad k=0,1,2,\dots\end{equation}
with $p_{-1}(s):=0$, $a_k>0$, and $b_k\in\R$.

Let assume that the $\sigma$-algebra $\mathscr A$ from the probability space $(\Omega,\mathscr A,P)$ is the minimal $\sigma$-algebra with respect to which $\gamma(A)$ is measurable for each $A\in\mathscr B_0(X)$. We denote by $\mathscr {CP}$ the {\it set of continuous polynomials of $\gamma$}, i.e., the set of random variables of the form
\begin{equation}\label{vtfr75r}
f^{(0)}+\sum_{i=1}^n\langle f^{(i)},\gamma^{\otimes i}\rangle,
\end{equation}
where $f^{(i)}\in\mathscr D^{\odot_{\mathrm{a}}i}$ and $n\in\mathbb N$. If $f^{(n)}\ne0$, we call the random variable in \eqref{vtfr75r} a continuous polynomial of $\gamma$ of degree $n$. Condition \eqref{dre64e} implies that $\mathscr {CP}$ is a dense subset of $L^2(\Omega,P)$.

Let us denote $\mathscr {CP}_n$ the subset of $\mathscr {CP}$ which consists of all polynomials of $\gamma$ of degree $\le n$. Let ${\mathscr {MP}}_n$ denote the closure of $\mathscr {CP}_n$ in  $L^2(\Omega,P)$ ({\it measurable polynomials of degree $\le n$}). Let $\mathscr{OP}_n:=\mathscr {MP}_n\ominus\mathscr {MP}_{n-1}$, where $\ominus$ means the orthogonal difference in $L^2(\Omega,P)$ ({\it orthogonal polynomials of degree $n$}). As a result, we get the orthogonal decomposition $L^2(\Omega,P)=\bigoplus_{n=0}^\infty \mathscr{OP}_n$.
For $f^{(n)}\in\mathscr D^{\odot_{\mathrm{a}}n}$, we denote by $P(f^{(n)})$ the orthogonal projection of the monomial $\langle f^{(n)},\gamma^{\otimes n}\rangle$
onto $ \mathscr{OP}_n$. Note that $P(f^{(n)})$ does not need to belong to $\mathscr {CP}$. For $F\in\mathscr F_{\mathrm{fin}}(\mathscr D)$ as in \eqref{drte6u}, we denote $P(F):=\sum_{i=0}^n P(f^{(i)})$.
The set of all such random variables we denote by $\mathscr{OCP}$ ({\it orthogonalized continuous polynomials}).

\begin{theorem}\label{dr5e5}
Let $\gamma$ be a measure-valued L\'evy process that satisfies the above assumptions. We have $\mathscr{OCP}\subset \mathscr{CP}$ (and, in fact, $\mathscr{OCP}= \mathscr{CP}$) if and only is there exist constants $\eta\ge0$ and $\lambda>0$ with $\lambda\ge 2\sqrt\eta$ such that: if $\eta=0$ then $\zeta=\lambda^{-2}\delta_\lambda$; and if  $\eta>0$ then   the measure $\zeta'$ has infinitely many points in its support and the system $(p_k)_{k=0}^\infty$ of monic polynomials that are orthogonal with respect to $\zeta'$ satisfies the recurrence relation   \eqref{tre6} with $a_k=\eta k(k+1)$ and $b_k=\lambda(k+1)$. Furthermore, for any $\eta$ and  $\lambda$ as above, we have, for $\varphi\in\mathscr D$ and $F\in\mathscr F_{\mathrm{fin}}(\mathscr D)$,
$$ \langle \varphi,\gamma\rangle P(F)=P\left(\left[
a^+(\varphi)+\lambda a^0(\varphi)+a_1^-(\varphi)+\eta a^-_2(\varphi)+2\big(\lambda+\sqrt{\lambda^2-4\eta}\big)^{-1}\right]F\right),$$
where the operators $a^+(\varphi)$, $a^0(\varphi)$, and $a^-_1(\varphi)$, and  $a^-_2(\varphi)$  were defined in Proposition~\ref{yur75ri7}.
\end{theorem}

\begin{remark}\label{dye7i} If $\eta>0$ and $\lambda>2\sqrt\eta$, we get from Theorem~\ref{dr5e5} and e.g.\ \cite[Chapter VI, Section~3]{Chihara} that the L\'evy measure of the corresponding measure-valued L\'evy process $\gamma$ is
$$\zeta=\frac1\eta\sum_{k=1}^\infty\bigg(\frac{\eta}{\upsilon^2}\bigg)^k\frac1k\,\delta_{(\upsilon-\frac\eta\upsilon)k}\,,$$
with
$$\upsilon:=(\lambda+\sqrt{\lambda^2-4\eta})/2,$$
 and for $A\in\mathscr B_0(X)$, the distribution of the random variable $\gamma(A)$ is the negative binomial distribution
 $$\bigg(\frac{\upsilon^2-\eta}{\upsilon^2}\bigg)^{ m (A)/\eta}\sum_{k=0}^\infty\bigg(\frac{\eta}{\upsilon^2}\bigg)^k\frac1{k!}\,\bigg(\frac{ m (A)}\eta\bigg)^{(k)}\,\delta_{(\upsilon-\frac\eta\upsilon)k}\,.$$
 Thus, $\gamma$ is a {\it negative binomial random measure}.

 In the case where $\eta>0$ and $\lambda=2\sqrt\eta$, the L\'evy measure of $\gamma$ is
 $$\zeta(ds)=\frac1{s\eta}\,e^{-\frac{s}{\sqrt\eta}}\,ds,$$
and for $A\in\mathscr B_0(X)$, the distribution of the random variable $\gamma(A)$ is the gamma distribution
$$\left[\eta^{\frac{ m (A)}{2\eta}}\Gamma\bigg(\frac{ m (A)}\eta\bigg)\right]^{-1}u^{\frac{ m (\Delta)}\eta-1}\, e^{-\frac u{\sqrt\eta}}\,\chi_{\R_+}(u)\, du.$$
Thus, $\gamma$ in this case is the {\it gamma random measure}, see e.g.\ \cite{KSSU,KL,TsVY}. The Laplace transform of $\gamma$ can also be written in the form
\begin{equation}\label{vyr7}
\mathbb E\left(e^{\langle f,\gamma\rangle}\right)=\exp\bigg[-\frac1\eta\int_X\log\big(1-\sqrt\eta\, f(x)\big)\, m (dx)\bigg],\quad f\in C_0(X),\
 f\le0.
\end{equation}
\end{remark}

\begin{proof}[Proof of Theorem \ref{cyd5aq42q}] Let
$\eta>0$. Set $\lambda=\beta+\beta^{-1}\eta$. We have
$$2\big(\lambda+\sqrt{\lambda^2-4\eta}\big)^{-1}=\beta^{-1}.$$
As easily seen, $\beta>\sqrt\eta$. Hence,
$$\lambda=\beta+\frac\eta\beta>2\sqrt\eta.$$
Let $\gamma_{\lambda,\eta}$ be the measure-valued L\'evy process on $X$ from Theorem~\ref{dr5e5} that corresponds to the parameters $\lambda$, $\eta$. By Proposition~\ref{yur75ri7} and  Theorem~\ref{dr5e5}, we get
\begin{equation}\label{rts5w}
\tau\big(R(f_1)\dotsm R(f_n)\big)=\mathbb E\big(\langle f_1,\gamma_{\lambda,\eta}\rangle\dotsm \langle f_n,\gamma_{\lambda,\eta}\rangle\big),\quad f_1,\dots,f_n\in\mathscr D.\end{equation}
We define the measure-valued L\'evy process $\gamma:=\kappa\sqrt{1+\eta\kappa^2}\,\gamma_{\lambda,\eta}$. Let $\zeta$ and $\zeta_{\lambda,\eta}$ denote the L\'evy measure of $\gamma$ and $\gamma_{\lambda,\eta}$, respectively. Then $\zeta$ is the pushforward of the measure  $\zeta_{\lambda,\eta}$ under the mapping $\R_+\ni s\mapsto \kappa\sqrt{1+\eta\kappa^2}\,s\in\R_+$. By \eqref{vcyre6jr876r7} and  \eqref{rts5w}, we get \eqref{cye57}. Formulas \eqref{crte64}, \eqref{yur7e65ew} follow by direct calculations from Remark~\ref{dye7i}. Note that, since the L\'evy measure $\zeta$ is concentrated on $\mathbb N$, $\gamma$ is a point process. The proof for $\eta=0$ is analogous.

Let us now prove that $\tau$ is not positive definite for $\eta<0$. Assume the contrary. Let $f\in\mathscr D$. We easily calculate
$$\tau\left(W(f\otimes f)^2\right)=2\bigg(\int_X f^2(x)\, m (dx)\bigg)^2+2\eta\int_Xf^4(x)\, m (dx)\ge0.$$
From here we conclude by approximation that, for any cube $A$ in $X$,
$$ m (A)^2+\eta m (A)\ge0,$$
which obviously fails for $A$ small enough.
\end{proof}

\begin{remark}
In view of Theorem \ref{cyd5aq42q}, we can rigorously understand $\rho(\varphi)$ ($\varphi\in\mathscr D$) as the operator of multiplication by $\langle\varphi,\gamma\rangle$ in the Hilbert space $L^2(\Omega,P)$, which maps  $\mathscr {CP}$  into itself.
It follows from \cite[Lemma~2.1 and Theorem~2.1]{Lyt} that each $\rho(\varphi)$ is essentially self-adjoint on $\mathscr {CP}$.
\end{remark}

\begin{corollary}Let $\eta>0$. Denote by $\tau_\kappa$ the state $\tau$ on $\mathfrak R$ which corresponds to the operator $T=\kappa^2\mathbf 1$. Let $\gamma$ be the gamma random measure with Laplace transform \eqref{vyr7}. Then, for any $f_1,\dots,f_n\in\mathscr D$, we have
\begin{align*}\lim_{\kappa\to\infty}\tau_\kappa\big(R(f_1)\dotsm R(f_n)\big)&=
\lim_{\kappa\to\infty}\left(\kappa\sqrt{1+\eta\kappa^2}\right)^{-n}\tau_\kappa\big(\rho(f_1)\dotsm \rho(f_n)\big)\\
&=\mathbb E\big(\langle f_1,\gamma\rangle\dotsm \langle f_n,\gamma\rangle\big).\end{align*}
\end{corollary}

\begin{proof}The statement follows immediately  from Proposition~\ref{yur75ri7}, Theorem~\ref{dr5e5}, and Remark~\ref{dye7i} if we note that
$$\lim_{\kappa\to\infty}\left((\eta+\kappa^{-2})^{1/2}+(\eta+\kappa^{-2})^{-1/2}\eta\right)=2\sqrt\eta,\quad  \lim_{\kappa\to\infty}(\eta+\kappa^{-2})^{-1/2}=\eta^{-1/2}.$$
\end{proof}

\begin{center}
{\bf Acknowledgements}\end{center}
The author acknowledges the financial support of the the SFB 701 ``Spectral
structures and topological methods in mathematics'', Bielefeld University, and of the 
Polish
National Science Center, grant no.\ Dec-2012/05/B/ST1/00626.  I am grateful to M.~Bo\.zejko, G.A. Goldin, Y. Kondratiev, A. Vershik, and J. Wysocza\'nski for useful discussions.
I  would  like to thank the referees for a careful reading of the
manuscript and making useful comments and suggestions.

\end{document}